\documentclass[11pt]{article}
\usepackage[utf8]{inputenc}
\usepackage[T1]{fontenc}
\usepackage[a4paper, margin=2.7cm]{geometry}
\usepackage{amsmath, amsthm, amsfonts, amssymb}
\usepackage{mathrsfs}          
\usepackage[colorlinks=true,linkcolor=blue,citecolor=blue,urlcolor=blue,breaklinks]{hyperref}
\usepackage{bbm} 
\usepackage{dsfont}
\usepackage{enumitem}
\usepackage{hyperref}
\usepackage{breakurl}
\usepackage[square,sort,comma,numbers]{natbib}
\usepackage{url}
\def\R{\mathbb{R}}

\def\1{\mathbbm{1}}

\def\d{\,\mathrm{d}}

\def \ddt{\frac{\mathrm{d}}{\mathrm{d}t}}
\def \ddx{\frac{\mathrm{d}}{\mathrm{d}x}}

\def\p{\partial}
\def\:{\colon}

\newtheorem{thm}{Theorem}[section]
\newtheorem{cor}[thm]{Corollary}
\newtheorem{lem}[thm]{Lemma}

\newtheorem{hyp}{Hypothesis}[section]

\theoremstyle{definition}

\theoremstyle{remark}
\newtheorem{rem}[thm]{Remark}
\theoremstyle{example}

\hypersetup{pdftitle={Growth-fragmentation equation}}
\hypersetup{pdfauthor={José A. Cañizo  \& Pierre Gabriel \& Havva Yoldaş}}

\title{Spectral gap for the growth-fragmentation equation via Harris's Theorem}

\author{José A. Cañizo\footnote{\textsc{Departamento de
			Matemática Aplicada, Universidad de Granada, 18071 Granada,
			Spain.}
		\textit{E-mail address}: \texttt{\href{mailto:canizo@ugr.es}{canizo@ugr.es}}}
	\and Pierre Gabriel\footnote{\textsc{Laboratoire de Mathématiques de Versailles, UVSQ, CNRS, Université Paris-Saclay, 45 Avenue des \'Etats-Unis, 78035 Versailles Cedex, France.} 
		\textit{E-mail address}: \texttt{\href{mailto:pierre.gabriel@uvsq.fr}{pierre.gabriel@uvsq.fr}}}
	\and Havva Yolda\c{s}\footnote{\textsc{Institut Camille Jordan, Universit\'e Claude Bernard Lyon 1, B\^atiment Jean Braconnier, 21 Avenue Claude Bernard, 69622 Villeurbanne Cedex, France.}
		\textit{E-mail address}: \texttt{\href{mailto:yoldas@math.univ-lyon1.fr}{yoldas@math.univ-lyon1.fr}}}
}

\begin{document}
	
	\maketitle
	
	\begin{abstract}
		We study the long-time behaviour of the growth-fragmentation
		equation, a nonlocal linear evolution equation describing a wide
		range of phenomena in structured population dynamics. We show the
		existence of a spectral gap under conditions that generalise those
		in the literature by using a method based on Harris's theorem, a
		result coming from the study of equilibration of Markov
		processes. The difficulty posed by the non-conservativeness of the
		equation is overcome by performing an $h$-transform, after solving
		the dual Perron eigenvalue problem.  The existence of the direct
		Perron eigenvector is then a consequence of our methods, which prove
		exponential contraction of the evolution equation.
		Moreover the rate of convergence is explicitly quantifiable in terms of
		the dual eigenfunction and the coefficients of the equation.
	\end{abstract}
	
	\tableofcontents
	
	\section{Introduction and main result}
	\label{sec:intro}
	
	The growth-fragmentation equation is a linear, partial
	integro-differential equation which is commonly used in structured population
	dynamics for modelling various phenomena including the time evolution
	of cell populations in biology such as in \cite{APM03, B68, BA67,
		BCP08, DHT84, HW89, DM14, P06, RTK17}, single species populations
	\cite{SS71}, or carbon content in a forest \cite{BGP19}; some
	aggregation and growth phenomena in physics or biophysics as in
	\cite{BLL19,CLODLMP,EPW,G15,LW17,McGZ87}; neuroscience in
	\cite{CY19,PPS14} and even TCP/IP communication protocols such as in
	\cite{BMR02, BCGMZ13,CMP10}. The general form of the growth-fragmentation
	equation is given by:
	\begin{equation}
		\begin{aligned} \label{eq:gf}
			\frac{\p }{\p t}n(t,x) + \frac{\p }{\p x} (g(x) n(t,x)) + B(x)n(t,x) &= \int_{x}^{+ \infty}  \kappa (y,x) n(t,y) \d y,  &&t,x > 0,\\
			n(t,0) &= 0,  &&t \geq 0,\\
			n(0,x) &= n_0(x),   &&x> 0,
		\end{aligned} 
	\end{equation}
	where $n(t,x)$ represents the population density of individuals
	structured by a variable $x > 0$ at a time $t \geq 0$. The structuring
	variable $x$ could be \emph{age, size, length, weight, DNA content,
		biochemical composition} etc. depending on the modelling
	context. Here we refer to it as \emph{`size'} for simplicity. Equation
	\eqref{eq:gf} is coupled with an initial condition $n_0(x)$ at time
	$t=0$ and a Dirichlet boundary condition which models the fact that no
	individuals are newly created at size $0$. The function $g$ is the
	\emph{growth rate} and $B$ is the \emph{total division/fragmentation
		rate} of individuals of size $x \geq 0$. The fragmentation kernel
	$\kappa(y,x)$ is the rate at which individuals of size $x$ are
	obtained as the result of a fragmentation event of an individual of
	size $y$. When fixing $x$, $\kappa(x, \cdot)$ is a nonnegative measure on
	$(0,x]$. The \emph{total fragmentation rate} $B$ is always obtained as
	\begin{equation}\label{eq:kappaB}
		B(x) = \int_0^y \frac{y}{x} \kappa(x,y) \d y,
		\qquad x > 0.
	\end{equation}
	Important particular cases are
	\begin{equation*}
		\kappa(x,y) = B(x) \frac{2}{x} \delta_{\{y=\frac{x}{2}\}},
	\end{equation*}
	which corresponds to the \emph{mitosis} process, suitable for
	modelling of biological cells, where individuals can only break into
	two equal fragments; and
	\begin{equation*}
		\kappa(x,y) = B(x) \frac{2}{x},
	\end{equation*}
	which is the case with \emph{uniform fragment distribution}, where
	fragmentation gives fragments of any size less than the original one
	with equal probability. This case is used for example in modelling the dynamics of polymer chains, as in \cite{EPW}.
	
	Two opposing dynamics, growth and fragmentation, are balanced through
	Equation~\eqref{eq:gf}. The growth term tends to increase the average
	size of the population and the fragmentation term increases the total
	number of individuals but breaks the population into smaller sizes. If
	the growth rate $g(x)$ vanishes, then only fragmentation takes place
	and the equation is known as the \emph{pure fragmentation
		equation}. Similarly when $B$ and $\kappa$ are both $0$, Equation \eqref{eq:gf} is the \emph{pure growth equation}.
	
	We are concerned here with the mathematical theory of this equation,
	and more precisely with its long-time behaviour as $t \to
	+\infty$. Under suitable conditions on the coefficients $\kappa$ and
	$g$, the typical behaviour is that the total population tends to grow
	exponentially at a rate $e^{\lambda t}$, for some $\lambda > 0$, and
	the normalised population distribution tends to approach a
	\emph{universal profile} for large times, independently of the initial
	condition. This has been investigated in a large amount of previous
	works, of which we give a short summary. The first mathematical study
	of this type of equation was done in \cite{DHT84} for the mitosis
	case, in a work inspired by some biophysical papers \cite{B68,
		BA67,SS71}. In \cite{DHT84}, the authors considered the mitosis
	kernel with the size variable in a bounded interval and proved
	exponential growth at a rate $\lambda$, and exponentially fast
	approach to the universal profile. In \cite{MMP05}, the authors
	considered the size variable in $(0,+\infty)$ and introduced the
	\emph{general relative entropy} method for several linear PDEs
	including the growth-fragmentation equation. They proved relaxation to
	equilibrium in $L^p$ spaces without an explicit rate. Following
	\cite{PR05} and \cite{LP09}, providing an explicit rate of convergence
	to a universal profile under reasonable assumptions became a topic of
	research for many other works. New functional inequalities were proved
	in \cite{CCM11, CCM11-2} in order to obtain explicit rates of
	convergence, see also~\cite{GS14}. Some authors used a
	semigroup approach \cite{O92, BA06, BPR12, BG18, EN01, GN88,MS16}
	or a probabilistic approach \cite{BCGMZ13,B03, B19, BW16, BW18, BW20, B18, BGP19,C21,C20,CMP10,C17},
	and some authors provided explicit
	solutions as in~\cite{ZvBW152}. In this paper we are
	able to give more general results regarding the speed of convergence
	to equilibrium: we obtain constructive results which cover a wide
	range of bounded and unbounded fragmentation rates, and which apply
	both in mitosis and uniform fragmentation situations.
	
	When the equal mitosis kernel is considered, there is a special case
	with a linear growth rate where the solutions exhibit oscillatory
	behaviour in long time. This property was first proved mathematically
	in \cite{GN88} when the equation is posed in a compact set. Recently,
	this result was extended to $(0,+\infty)$ by the general relative
	entropy argument in suitable weighted $L^2$ or measure spaces in
	\cite{BDG18,GM19} and by means of Mellin transform in $L^1$ space by~\cite{vBALZ}.
	
	\medskip
	
	An important tool when studying the asymptotic behaviour of
	\eqref{eq:gf} is the Perron eigenvalue problem: finding a positive
	eigenfunction for the operator which defines the equation, associated
	to a simple, real eigenvalue which is also equal to the spectral
	radius; see \cite{M06,DG09} for general existence results. In
	\cite{BCG13}, the authors gave some estimates on the principal
	eigenfunctions of the growth-fragmentation operator, giving their
	first order behaviour close to $0$ and $+\infty$. Then they proved a
	spectral gap result by means of entropy–entropy dissipation
	inequalities, with tools similar to those of~\cite{CCM11,
		CCM11-2}. They assumed that the growth and the fragmentation coefficients
	behave asymptotically like power laws.
	
	\medskip In this paper we use a probabilistic approach, namely
	\emph{Harris's theorem}, for showing the spectral gap
	property. We give a novel approach based on
		estimating solutions to the PDE, and obtain results which
		can be applied to general growth and fragmentation rates
		including mitosis and uniform fragmentation cases. Detailed
		hypotheses and results are given later in this
		introduction. The method is also completely constructive and
		gives explicit estimates. However, in some cases these
		estimates depend on estimates on the first dual
		eigenfunction, which may be not easy to obtain, but
		constitute a separate question. After stating our results we
		also give a brief comparison to other spectral gap results
		in the literature.
	
	Applications of this type of argument into biological and
	kinetic models which can be described as \emph{Markov
		processes} is becoming a subject of many works recently, and has been extended to models which are not
		Markov processes but share similar properties. The
	predecessor of Harris's theorem, namely \emph{Doeblin's
		argument} is used in \cite{G18} for proving exponential
	relaxation of solutions to the equilibrium for the
	conservative renewal equation. In \cite{CY19} and \cite{DG17},
	the authors study population models which describe the
	dynamics of interacting neurons, structured by elapsed-time
	in~\cite{CY19} or by voltage in~\cite{DG17}, and existence of
	a spectral gap property in the \textit{`no-connectivity'}
	setting is proved by Doeblin's Theorem. Moreover, there are
	some recent works for the extension of this method into the
	non-conservative setting. In \cite{BCG17}, the authors
	consider several types of linear PDEs including a
	growth-diffusion model with time-space varying environment and
	some renewal equations with time-fluctuating ({\it e.g.}
	periodic) coefficients. They provide quantitative estimates in
	total variation distance for the associated non-conservative
	and non-homogeneous semigroups by means of generalized
	Doeblin's conditions. The full Harris's theorem is used in
	\cite{B18,BGP19} for deriving exponential convergence to the
	equilibrium in the conservative form of the
	growth-fragmentation equation.  In the present work, we are
	interested in the long time behaviour of the more challenging
	non-conservative case, namely when no quantity is preserved
	along time.  Our method is in the spirit of~\cite{BCGM19},
	where a non-conservative version of Harris's Theorem is
	proposed and applied to the growth-fragmentation equation with
	constant growth rate $g$ and increasing total division rate
	$B$, see also~\cite{CG20} for an application to a
	mutation-selection model which is similar to growth-fragmentation.
	The difference here is that we first build a solution to
	the dual Perron eigenproblem by using Krein-Rutman's theorem
	and a maximum principle.  Then we take advantage of the dual
	eigenfunction to perform a so-called (Doob) $h$-transform
	\cite{D57}, similarly as in~\cite{BPR12,C17}, in order to
	apply Harris's theorem.  It allows us to consider very general
	growth and fragmentation rates.  The drawback is that the
	spectral gap is given explicitly in terms of the dual
	eigenfunction, for which quantitative estimates are in general
	hard to obtain.  However, for certain specific coefficients
	that are worth of interest, the dual eigenfunction is known
	explicitly.  It is the case of the so-called
	\emph{self-similar fragmentation equation}, widely studied in
	the literature, for which we provide new quantitative
	estimates on the spectral gap.
	
	\
	
		Let us now precise the functional analytic
		setting of our work and what we mean by solutions to
		Equation~\eqref{eq:gf}.  We are interested in measure
		solutions to this equation, which is a relevant notion in
		population dynamics, see {\it e.g.}~\cite{CCC13,G18}.  We
		say that a family $(n(t,\cdot))_{t\geq0}$ of positive
		measures on $(0,+\infty)$ is a solution to
		Equation~\eqref{eq:gf} if for all $f\in C^1_c([0,+\infty))$
		the function $t\mapsto\langle n(t,\cdot),f\rangle$ is
		continuously differentiable and for all $t\geq0$
		\begin{equation}\label{eq:gf_def}\frac{\d}{\d t}\langle n(t,\cdot),f\rangle=\langle n(t,\cdot),\mathcal L^*[f]\rangle,\end{equation}
		where
		\[\mathcal{L}^*[f](x) := g(x) \frac{\p}{\p x}f (x) +
		\int_{0}^{x} \kappa (x,y)f(y) \d y -B(x) f(x)\]
		is the dual operator of the growth-fragmentation operator
		\[\mathcal{L} [n] (x) := - \frac{\p}{\p x} (g(x)n(x)) - B(x) n(x) + \int_{x}^{+\infty} \kappa (y,x) n(y) \d y, \]
		which appears in Equation~\eqref{eq:gf}.  We refer
		to~\cite{BCGM19,GM19} for (the method of) proof that
		Equation~\eqref{eq:gf} is well-posed in the set of positive
		(or signed Radon) measures $\mu$ such that the weighted total
		variation norm
		\begin{equation}\label{eq:wtv}\left\| \mu \right\|_{V} = \int_{0}^{+\infty} V(x) |\mu|(\mathrm{d}x)\end{equation}
		is finite, when $V(x)=x^k+x^K$ with $k\leq0$ and $K>1$.
		
		\
		
		The Perron eigenvalue problem consists of finding suitable
		eigenelements $(\lambda, N, \phi )$ with $\lambda > 0$ and
		$N, \phi \: (0,+\infty) \to [0,+\infty)$, $N,\phi\not\equiv0$, satisfying the following:
		\begin{equation} \label{eq:eigenfunction}
			\mathcal{L}[N]=\lambda N,\quad(gN)(0) = 0,
		\end{equation}
		\begin{equation} \label{eq:dualeigenfunction}
			\mathcal{L}^*[\phi]=\lambda\phi. 
	\end{equation}
	If such a triple exists then $\lambda$ is actually the
	dominant eigenvalue of Equation \eqref{eq:gf}, and the solution is expected to converge to a universal profile whose shape is given by the eigenfunction $N(x)$. The
	convergence rate is given by the gap between the dominant eigenvalue
	$\lambda >0$ and the rest of the spectrum. If we scale the
	equation by defining $m(t,x) := n(t,x) e^{-\lambda t}$ we obtain:
	\begin{equation}
		\begin{aligned}
			\label{eq:gfscaled}
			\frac{\p }{\p t}m(t,x) + \frac{\p }{\p x} (g(x) m(t,x)) + (B(x) + \lambda) m(t,x)&= \int_{x}^{+ \infty}  \kappa (y,x) m(t,y) \d y,  &&t,x \geq 0,\\
			m(t,0) &= 0, \qquad \qquad &&t > 0,\\
			m(0,x) &=  n_0(x), \qquad  &&x > 0.
		\end{aligned}
	\end{equation}
	We remark that $N(x)$ is the stationary solution of Equation \eqref{eq:gfscaled}
	and $\phi(x)$ provides a conservation law for \eqref{eq:gfscaled} since
		\[\frac{\d}{\d t}\int_{0}^{+\infty}\phi(x) m(t,x) \d x =0. \]
	Since the existence and
	uniqueness of the eigenelements provide useful information about the long
	time behaviour of the growth-fragmentation equation \eqref{eq:gf}, it
	has been a popular topic of research. We refer to \cite{DG09}
	for a general recent result. From now on we consider
	Equation \eqref{eq:gfscaled} instead of Equation \eqref{eq:gf} since it is more
	convenient to study the long-time behaviour of the former and we
	can easily recover the nature of the latter.
	
	\
	
	We now list all the assumptions we need throughout the paper. 
	
	\
	
	As we will explain in Section~\ref{sec:harris},
		Harris's method relies on a local Doeblin's minorisation
		condition.  The computations for checking this condition
		strongly depend on the fragmentation kernel.  In~\cite{CY19}
		a global Doeblin condition is proved (for a similar
		equation) for kernels $\kappa$ which satisfy, for some
		$\epsilon,\eta,x_*>0$, the condition that
		$\kappa(x,y)\geq\epsilon$ for all $x\in[0,\eta]$ and
		$y\geq x_*$.  Here we rather consider kernels that are of
		self-similar form, which is commonly assumed in the
		literature about spectral gaps for the growth-fragmentation
		equation~\cite{BCG13,BCGM19,BG18,C21,CCM11,MS16} and
		includes the classical kernels appearing in applications (in
		particular equal or unequal mitosis and uniform fragment
		distribution, see below).  
	
	\begin{hyp}\label{asmp:k1}
		We assume that $\kappa(x,y)$, the fragmentation kernel, is of the
		self-similar form such that 
		\begin{equation*}  \label{hyp:kappa}
			\kappa (x,y) = \frac{1}{x}p\Big(\frac{y}{x} \Big) B(x), \quad \text{ for } y > x > 0,
		\end{equation*}
		where $p$, the ``fragment distribution'', is a nonnegative measure on $(0,1]$ such that $z p(z)$ is a probability measure; that is,
		\begin{equation*}
			\int_{(0,1]} z p(z) \d z = 1.
		\end{equation*}
	\end{hyp}
	
	\begin{rem}
		It is useful to define $p_k$, for $k\in\R$, as the $k$-th moment of
		$p$:
		\[p_k := \int_{0}^{1} z^k p(z) \d z.\] With this notation,
		Hypothesis~\ref{asmp:k1} ensures that $p_1 = 1$, so the
		relation~\eqref{eq:kappaB} is guaranteed.
	\end{rem}
	
	\medskip Our next hypothesis states that we consider only the two
	extreme cases of the fragment distribution, namely the very singular
	equal mitosis case and the very smooth uniform fragment
	distribution. One can find conditions for our methods to work in
	intermediate cases, but we have preferred to give simple proofs that
	show both singular and smooth cases can be treated:
	
	\begin{hyp} \label{asmp:p}
		We assume that the fragment distribution $p$ is either the one
		corresponding to the equal mitosis:
		\begin{equation}\label{eq:mitosis}
			p(\d z)=2\delta_{\frac12}(\d z)
		\end{equation}
		or the uniform fragment distribution:
		\begin{equation}\label{eq:uniform}
			p(\d z)=2\d z.
		\end{equation}
	\end{hyp}
	
	\begin{rem}\label{rk:kernel}
		We restrict to these two particular fragmentation kernels
		because they naturally appear in the modelling of natural
		phenomena. They are also good representatives of two
		opposite mathematical situations: a very regular, strictly
		positive case and a singular case which is positive only at
		$z=1/2$. However, the results which we prove to be valid for
		the uniform kernel can be readily extended to self-similar
		kernels with $p$ satisfying
		\begin{equation}\label{asmp:plowerbound}
			p(z)\geq c>0 \qquad\text{for all
					$z$ in some interval $(z_1, z_2) \subseteq
					(0,1)$}
		\end{equation}
		and either
		\[p_0<+\infty\qquad \text{if}\qquad\int_0^1\frac{1}{g(x)}\d x<+\infty,\]
		or
		\[\exists k<0 \ \text{ with }\  p_k< +\infty\qquad \text{if}\qquad\int_0^1\frac{1}{g(x)}\d x=+\infty.\]
		In the particular case of the linear growth rate, $g(x) = x$, it is enough to assume that
		\[\exists k<1 \ \text{ with }\ p_k< +\infty.\]
		Notice that under condition~\eqref{asmp:plowerbound}, similarly as
		for~\eqref{eq:mitosis} and~\eqref{eq:uniform}, the function
		$k\mapsto p_k$ is strictly decreasing on the interval where it takes
		finite values. (The only case in which $p_k$ is not strictly
		decreasing is that of $p(z)$ concentrated at $z=1$, which actually
		means no fragmentation at all is happening.)
		
			In the case of constant growth rate, a more
			general condition than~\eqref{asmp:plowerbound} is assumed
			in~\cite{BCGM19} that covers the unequal mitosis kernels
			$p(\d z)=\delta_\alpha(\d z)+\delta_{1-\alpha}(\d z)$ with
			$0<\alpha<1$. In our proofs we can also consider this
			generalisation with straightforward modifications when the
			growth rate $g$ satisfies forthcoming
			Hypothesis~\ref{asmp:gp}.
			
			Regarding non self-similar kernels, there are results of
			exponential convergence to the stationary distribution in
			the literature, but only for bounded fragmentation rates;
			see~\cite{LP09,PPS14} for PDE-based arguments and
			\cite{B19,BW18, C21} for a probabilistic point of view. We
			also point out that an optimal condition on the fragment
			distribution is given in \cite{B19} for a spectral gap to
			exist (for bounded fragmentation rates).
	\end{rem}

	Next we have a general assumption on the growth rate $g$ and the total
	fragmentation rate~$B$:
	
	\begin{hyp}\label{asmp:gB}
		We assume that $g \: (0,+\infty) \to (0,+\infty)$ is a locally
		Lipschitz function such that $g(x) =\mathcal{O}(x)$ as
		$x\to+\infty$ and $g(x)=\mathcal{O}(x^{-\xi})$ as $x\to0$ for some
		$\xi\geq 0$.  The total fragmentation rate
		$B\: [0,+\infty) \to [0,+\infty)$ is a continuous function and the
		following holds
		\begin{equation}
			\label{eq:gB}
			\int_0^1\frac{B(x)}{g(x)}\d x<+\infty,
			\qquad\frac{xB(x)}{g(x)}\underset{x \to 0}{\longrightarrow } 0,
			\qquad \frac{xB(x)}{g(x)} \underset{x \to +\infty}{\longrightarrow } + \infty.
		\end{equation}
	\end{hyp}
	
	This assumption is very mild, and is always present in the previous works
	to ensure the existence of an equilibrium and a dual eigenfunction. If
	$B$ behaves like a power of exponent $b$ and $g$ behaves like a
	power of exponent $a$, conditions \eqref{eq:gB} are equivalent to
	the more familiar $b - a + 1 > 0$. The condition
	$g = \mathcal{O}(x)$ for large $x$ ensures that the characteristics
	corresponding to the growth part are defined for all times (i.e.,
	clusters do not grow to infinite size in finite time). A stronger
	assumption which is implicit in Hypothesis \ref{asmp:gB} is that $B$ is bounded
	above on intervals of the form $[0,R]$ (since it is continuous there),
	so we do not allow fragmentation rates $B$ which blow up at $0$. This
	is used in the proof of Lemma \ref{lem:doeblinuniformfrag}.
	
	\smallskip
	A consequence of Hypothesis \ref{asmp:gB}, later we will need 
	the following:
	\begin{equation}
		\label{eq:tB}
		\parbox{.8\linewidth}{There exists $t_B > 0$ such that $B$ is bounded below by a
			positive quantity on any interval of the form $[t_B,\theta]$ with
			$\theta > t_B$.}
	\end{equation}
	One sees this from the last limit in \eqref{eq:gB}, which implies that
	for large enough $t_B$ we have
	\begin{equation*}
		B(x) \geq \frac{g(x)}{x}.
	\end{equation*}
	This easily implies \eqref{eq:tB}, since $g(x)/x$ is continuous and
	strictly positive, so bounded below by some positive quantity on any
	compact interval.
	
	\
	
	Our last assumption gives a stronger requirement on the growth rate
	$g$ when the mitosis kernel is considered. In this case, some additional requirement
	is necessary, since when the linear growth rate with equal
	mitosis is considered, it is known that there is no spectral gap
	\citep{BDG18,GM19,vBALZ}. We point out that the sharp assumption of ``there
	exists a point $x > 0$ with $g(2x)\neq 2g(x)$'' is enough to show
	convergence to the profile $N$, without a rate and only in particular cases,
	as proved in \cite[Section 6.3.3]{RTK17}. Our assumption is stronger
	than this, but also leads to a stronger result:
	
	\begin{hyp}\label{asmp:gp}
		When $p$ is the equal mitosis kernel~\eqref{eq:mitosis}, we assume
		that the growth rate $g$ satisfies
		\begin{gather*}
			\omega g(x) < g(\omega x)
			\qquad \text{for all $x > 0$ and $\omega \in (0,1)$,}
			\\
			H(z) := \int_0^z \frac{1}{g(x)} \d x < +\infty
			\qquad \text{for all $z > 0$,}
		\end{gather*}
		and also $H^{-1}$ (the inverse of $H$) does not grow too fast, in
		the sense that for all $r > 0$ we have
		\begin{equation}
			\label{eq:H-1-power}
			\lim_{z \to +\infty} \frac{H^{-1}(z + r)}{H^{-1}(z)} = 1.
		\end{equation}
	\end{hyp}
	
	If we consider just powers, examples of growth and fragmentation
	rates which satisfy all of the above are
	\begin{gather*}
		B(x) = x^b,\qquad  g(x) = x^a
	\end{gather*}
	with:
	\begin{itemize}
		\item any $b \geq 0$, $-\infty < a \leq 1$ in the uniform
		fragment distribution case, excluding the case
		$(b, a) = (0, 1)$.
		\item any $b \geq 0$, $-\infty < a < 1$ in the mitosis case.
	\end{itemize}

	\medskip
	
	Under Hypothesis \ref{asmp:k1}, the rescaled
	growth-fragmentation equation \eqref{eq:gfscaled} takes the
	form:
	\begin{align} \label{eq:gfscaledgeneral}
		\begin{split}
			\frac{\p }{\p t}m(t,x) +  \frac{\p }{\p x} (g(x)m(t,x)) + c(x ) m(t,x) &= \mathcal{A}(t,x) , \hspace{5pt} t,x \geq 0,\\
			m(t,0)&= 0, \qquad \qquad t > 0,\\
			m(0,x) &= n_0(x), \qquad  x > 0.
		\end{split}
	\end{align} where 
	\[c(x) := B(x) + \lambda \] and
	\[\mathcal{A}(t,x) :=  \int_{x}^{+\infty}  \frac{B(y)}{y} p
	\left(\frac{x}{y}\right) m(t,y) \d y.
	\]
	According to Hypothesis \ref{asmp:p}, we only allow $p(z) = 2$ or
	$p(z) = 2 \delta_{\frac 1 2} (z)$.
	
	\
	
	Our main result is given by the following theorem:
	\begin{thm}
		\label{thm:main}
		Assume that Hypotheses \ref{asmp:k1}, \ref{asmp:p},
		\ref{asmp:gB}, and \ref{asmp:gp} are satisifed. Then there
		exists a solution $(\lambda,N,\phi)$ to the Perron
		eigenvalue
		problem~\eqref{eq:eigenfunction}-\eqref{eq:dualeigenfunction}
		with the normalization $\int N=\int\phi N=1$, $\lambda >0$,
		and there exist $C,\rho > 0$ such that the solution
		$n = n(t,x) \equiv n_t(x)$ to Equation \eqref{eq:gf} with initial
		data given by a nonnegative finite measure $n_0$ with
		$\|n_0\|_V < +\infty$ satisfies
		\begin{equation}\label{eq:convergence}
			\left\| e^{-\lambda t} n_t - \Big(\int\phi n_0\Big)N \right\|_{V}
			\leq
			C e^{-\rho t} \left\| n_0 - \Big(\int\phi n_0\Big)N\right\|_{V}
			\qquad \text{for all } t \geq 0,
		\end{equation}
		where the weight $V$ of the total variation norm $\|\cdot\|_V$ defined in~\eqref{eq:wtv} is given by
		\begin{equation*}
			\begin{aligned}
				V(x)&=1+x^K,\ 1+\xi <K\qquad
				&&\text{if}\qquad\int_0^1\frac{1}{g(x)}\d x<+\infty,
				\\
				V(x)&=x^k+x^K,\ -1<k<0, 1+\xi <K\qquad
				&&\text{if}\qquad\int_0^1\frac{1}{g(x)}\d x=+\infty.
			\end{aligned}
		\end{equation*}
		In the specific case of $g(x)=x$, the weight $V(x)$ can be taken to be
		\begin{equation*}
			V(x) = x^k+x^K,\ -1<k<1<K.
		\end{equation*}
	\end{thm}

	It is worth noticing that we obtain a spectral gap in spaces
	with essentially optimal weights.  Indeed it was proved
	in~\cite{BG17} that there is no spectral gap in weighted $L^1$ space
	with the dual eigenfunction $\phi$ when $B$ is bounded (see
	the estimates in Theorem~\ref{thm:dualeigenfunction} below).
	
	\medskip
	
	To the best of our knowledge, even the existence of the Perron eigenelements in
	such generality is new (allowing a total fragmentation rate with any
	growth at infinity, and with no required connectivity condition on its
	support), and hence so is the existence of a spectral gap. However,
	since our approach for the existence of the principal eigenfunction
	$N$ is a byproduct of the contraction result provided by Harris's
	theorem, this precludes the case of self-similar fragmentation with
	equal mitosis and growth rate $g(x)=x$, for which convergence to a
	universal profile does not hold, as we already mentioned. In that case
	the existence of a Perron eigenfunction has to be tackled with other
	spectral methods, as in~\cite{DG09,HW90,H85,M06}.
	
	\medskip
	
	Note also that our result is valid for the measure solutions
	of Equation~\eqref{eq:gf}, thus improving the result
	in~\cite{DDGW18} where the general relative entropy method is
	extended to measure solutions, providing convergence to
	Malthusian behaviour but without a rate and under restrictive
	assumptions on the coefficients.
	
	\medskip
	
	Regarding the assumptions on the coefficients,
		the only existing spectral gap results that consider general growth rates are the ones in~\cite{BCG13} and~\cite{BG18}.
		In theses papers, the fragmentation rate is assumed to behave like a power law,
		which we relax here by only requiring Hypothesis~\ref{asmp:gB} on $B$.
		The other results in the literature focus on constant or linear growth rates and,
		except in~\cite{BCGM19}, they also consider division rates that grow like power laws.

	\medskip
	
	Finally, when explicit estimates are available for
		$\phi$, our method allows us to derive quantitative estimates
	on the spectral gap.  It is the case for instance when $g(x)=x$
	since then $\phi(x)=x$.  An important particular case is to
	consider additionally that $B(x)=x^b$ for some $b>0$.  This
	corresponds to the so-called self-similar fragmentation
	equation, which appears as a rescaling of the pure fragmentation
	equation, see {\it e.g.}~\cite{DE16,EMRR05}.  To illustrate the
	quantification of the spectral gap, we prove that for the
	homogeneous fragmentation kernel and the choice $V(x)=1+x^2$,
	the inequality~\eqref{eq:convergence} holds true for
	\begin{equation}\label{rho_num}
		\rho=\frac{-\log\Big(1-\dfrac{\alpha}{2(1+2\alpha)}\Big)}{2\log2}
	\end{equation}
	where
	\[\alpha=2\log2\,R^{b+3}e^{-2(4R)^b/b}\quad \text{with}\quad R=80\Big(\frac{15}{2}\Big)^{\frac1b+\frac b2}.\]
	This seems to simplify the computable bound in~\cite[Proposition
	6.7]{MS16}.  It can also be compared to~\cite{GS14} where the
	spectral gap in $L^2(x\d x)$ is proved to be at least $\frac12$,
	but only for $b\geq2$. Similarly as
		in~\cite{MS16}, our method also allows for deriving explicit
		estimates for more general fragmentation kernels since it does
		not change the function $\phi$.
		
		Historically, the first explicit spectral gap was obtained for
		constant growth and division rates and the equal mitosis
		kernel in~\cite{PR05}, and then in~\cite{CMP10,BCGMZ13}.  The
		conditions were relaxed~\cite{LP09} and in particular general
		fragmentation kernels were considered.  Our method also allows
		to get explicit spectral gap in the case of constant growth
		rates, when the division rate is affine and the fragmentation
		kernel is self-similar.  Indeed if $g(x)=1$ and $B(x)=ax+b$,
		then we easily check that $\phi(x)=\alpha x+1$ with
		$\alpha=\frac{(p_0-1)b}{2}\big[\sqrt{1+\frac{4a}{(p_0-1)b^2}}-1\big]$,
		where we recall that $p_0$ is the mass of the self-similar
		kernel $p$, and the Perron eigenvalue is given by
		$\lambda=\frac{(p_0-1)b}{2}\big[\sqrt{1+\frac{4a}{(p_0-1)b^2}}+1\big]$.
		It is a particular case of the one treated in~\cite{BCGM19}
		where $B$ is only assumed to be non-increasing, but it extends
		the historical case of constant division rate.
	
	\
	
	This paper is organized as follows: 
	We devote Section \ref{sec:phi} to showing existence of the dual
	eigenfuction and some bounds on it. In Section \ref{sec:harris}, we
	recall some introductory concepts from the theory of Markov processes
	and state Harris's Theorem \ref{thm:Harris} based on
	the previous literature. Eventually for the proof of Theorem
	\ref{thm:main} which is given by applying Harris's theorem, we need to
	have Hypotheses \ref{hyp:Lyapunov} and \ref{hyp:localDoeblin}
	satisfied for Equation \eqref{eq:gfscaledgeneral}. In Sections
	\ref{sec:lyapunov} and \ref{sec:LowerBounds}, we prove that Hypotheses
	\ref{hyp:Lyapunov} and \ref{hyp:localDoeblin} are verified for Equation \eqref{eq:gfscaledgeneral}, respectively.
	Finally in Section~\ref{sec:proof-main} we give the proof of Theorem~\ref{thm:main} and the computations leading to~\eqref{rho_num}.
	
	\section{Existence of the dual eigenfunction}
	\label{sec:phi}
	
	In this section, we prove the following theorem which implies
	existence and boundedness of the dual Perron eigenfunction $\phi$, a
	solution to the dual eigenproblem \eqref{eq:dualeigenfunction}:
	
	\begin{thm}[Existence and bounds on the eigenfunction $\phi$]
		\label{thm:dualeigenfunction}
		We assume that Hypotheses \ref{asmp:k1} and \ref{asmp:gB} hold true
		and assume also that $p_0<+\infty$. Then there exist a continuous
		function $\phi$ which is a solution to Equation \eqref{eq:dualeigenfunction}
		and $C>0$ such that for any $k>1$;
		\[ 0 < \phi(x) \leq C(1+ x^k) \qquad \text{ for all } x>0. \]
		Additionally we have $\phi(0)>0$ when $\int_0^1\frac1g<+\infty$ and
		$\phi(0)=0$ when $\int_0^1\frac1g=+\infty$.
	\end{thm}
	Notice that our only assumption on $p$ is that $p_0 < +\infty$ (see
	Remark~\ref{rk:kernel}). We prove this theorem at the end of the
	section.
	
	\medskip 
	
	Following the idea introduced in~\cite{PR05} and also used
	in~\cite{BCG13,DG09}, we begin with defining a truncated version of
	the dual Perron eigenproblem \eqref{eq:dualeigenfunction} in an
	interval $[0,R]$ for some $R>0$:
	\begin{equation} 
		\begin{split}
			\label{eq:trdualeigenfunction1}
			-g(x) \frac{\p}{\p x } \phi_R(x) + (B(x) + \lambda_R ) \phi_R(x)
			= \frac{B(x)}{x} \int_{0}^{R} p\left(\frac{y}{x}\right)\phi_R(y) \d y,
			\\
			\phi_R(x) >0\quad\text{for}\ x\in(0,R),
			\qquad
			\phi_R(R) = 0.
		\end{split}
	\end{equation}
	\
	
	Now we give some lemmas which will be used in the proof of Theorem
	\ref{thm:dualeigenfunction}. The existence of a solution to Equation \eqref{eq:trdualeigenfunction1} is a consequence of the Krein-Rutman
	theorem (see the appendices in~\cite{DG09} and~\cite{BCG13}). Moreover
	in \cite{BCG13}, the authors proved that there exists $R_0>0$ large
	enough such that for all $R>R_0$ we have $\lambda_R >0$.  We thus have
	the following result:
	
	\begin{lem}\label{lem:existencetruncated}
		For any $R>0$, the truncated dual Perron eigenproblem
		\eqref{eq:trdualeigenfunction1} admits a solution
		$(\lambda_R,\phi_R)$ with $\phi_R$ a Lipschitz function.  Moreover
		there exists $R_0>0$ such that $\lambda_R>0$ for all $R>R_0$.
	\end{lem}
	
	Before proving uniform estimates on $(\lambda_R,\phi_R)$, we first
	recall a maximum principle. We begin by defining an operator
	$\mathcal{L}^*_R$, acting on once-differentiable functions
	$\varphi \in \mathcal{C}^{1}([0,R])$:
	\begin{equation*} \label{eq:max}
		\mathcal{L}^*_R \varphi(x):=  - g(x)\varphi'(x) +  \left( \lambda_R  + B(x) \right)\varphi(x)   - \frac{B(x)}{x}\int_{0}^{x}  p \left(\frac{y}{x} \right) \varphi(y)  \d y.
	\end{equation*} 
	We have the following maximum principle, see \cite[Appendix C]{DG09} or \cite[Section 3.2]{BCG13}:
	\begin{lem}
		\label{lem:maxprinciple}
		Suppose that  $ \varphi(x) \geq 0$ for $x \in [0,A]$ for some $A \in (0,R)$ with $\varphi(R) \geq 0$ and $\mathcal{L}^*_R \varphi(x) >0$ on $[A,R]$. Then $\varphi (x) \geq 0 $ on $[0,R]$.
	\end{lem}
	
	This maximum principle allows us to get a uniform upper bound on
	$\phi_R$, for a suitable normalization.
	
	\begin{lem}\label{lem:bounonphi}
		Consider that Hypotheses \ref{asmp:k1} and \ref{asmp:gB} are
		satisfied, and that $p_0<+\infty$.  For any $k>1$, there exists
		$A>0$ such that if $\phi_R$ is normalized such that
		\begin{equation} \label{norm}
			\underset{x \in [0,A]}{ \sup }  \, \, \phi_R (x)= 1,
		\end{equation}
		then for all $R>\max\{A,R_0\}$ and for all $x \in (0,R]$ we have  
		\[0 < \phi_R(x)\leq 1 +x^k.\]
		Additionally, $\phi_R(0)>0$ when $\int_0^1\frac1g<+\infty$ and $\phi_R(0)=0$ when $\int_0^1\frac1g=+\infty$.
	\end{lem}
	\begin{proof}
		For the bound from above we want to use the maximum principle in Lemma~\ref{lem:maxprinciple}. Therefore we want to prove that $\mathcal{L}^*_R \varphi(x) >0$ for $ x \in (A,R)$ with $A \in (0,R)$ as in Lemma~\ref{lem:maxprinciple}.
		We take $\varphi(x) = 1 + x^k$ for some $k>1$. Then for $R \geq R_0$ we have 
		\begin{align*} \label{psi}
			\begin{split}
				\mathcal{L}^*_R \varphi(x) &=  \lambda_R (1+x^k)  - kg(x) x^{k-1} + B(x) \left( (1 + x^k) - \frac{1}{x}\int_{0}^{x} (1 + y^k) p \left(\frac{y}{x}\right) \d y  \right) 
				\\ &= \lambda_R (1+x^k)  - kg(x) x^{k-1} + B(x) (1 + x^k - p_0 - x^{k}p_k) 
				\\ &>  x^{k-1} \left( -k g(x) -  B(x) x^{1-k} + (1-p_k) B(x)x \right) := \varrho (x)
			\end{split}
		\end{align*}
		since $p_0 = 2$ and $0< p_k < 1=p_1$ for $k>1$. 
		Moreover assuming \eqref{eq:gB} gives that behaviour of $\varrho$ will be dominated by the positive term $(1-p_k) B(x)x^k > 0$. 
		Therefore, we can find $A(k)>0$ such that for all $A(k) <x <R$, we have $\mathcal{L}^*_R \varphi(x) >0$.
		We fix such a $A> 0$ and normalize $\phi_R$ by~\eqref{norm}.
		Then by the maximum principle in Lemma~\ref{lem:maxprinciple} we obtain that $\phi_R(x)\leq1+x^k$.
		The positivity or nullity of $\phi_R(0)$ is a
		direct consequence of~\cite[Theorem 1.10]{BCG13}.
	\end{proof}

	\begin{lem}
		\label{lem:boundonlambda}
		Under Hypotheses \ref{asmp:k1} and \ref{asmp:gB} with $p_0<+\infty$,
		there exists a constant $C>0$ such that $\lambda_R \leq C$ for all
		$R >R_0$.
	\end{lem}
	
	\begin{proof}
		Since $\phi_R$ is continuous and by \eqref{norm}, there exists
		$x_R\in[0,A]$ such that $\phi_R(x_R)=1.$ Notice that necessarily
		$x_R>0$ when $\int_0^1\frac1g=+\infty$, since $\phi_R(0)=0$ is the
		case.  Moreover, the equation $\mathcal{L}^*_R \phi_R=0$ ensures that
		for all $x>0$ we have
		\begin{multline*}
			\left(\phi_R(x) \exp \left( -\int_{x_R}^x \frac{\lambda_R+B(s)}{g(s)} \d s\right) \right)' 
			\\ = - \frac{B(x)}{x g(x)} \exp \left( -\int_{x_R}^x \frac{\lambda_R+B(s)}{g(s)} \d s\right) \int_{0}^{x} p \left(\frac{y}{x}\right) \phi_R(y) \d y.
		\end{multline*} By integrating this from $x_R$ to $x\geq x_R$;
		\begin{align*}
			\phi_R(x) &\exp \left( -\int_{x_R}^x \frac{\lambda_R+B(s)}{g(s)} \d s\right)  - 1\\
			&= - \int_{x_R}^{x} \frac{B(y)}{y g(y)} \exp \left( -\int_{x_R}^y \frac{\lambda_R+B(s)}{g(s)} \d s\right) \int_{0}^{y} p \left(\frac{z}{y}\right) \phi_R(z) \d z \d y.
		\end{align*}
		By using the upper bound on $\phi_R$ we obtain, for $R> R_0$,
		\begin{align*}
			\phi_R(x)&\exp \left({-\int_{x_R}^x \frac{\lambda_R+B(s)}{g(s)} \d s} \right) 
			\\ &\geq 1 - \int_{x_R}^{x} \frac{B(y)}{y g(y)} \exp \left( -\int_{x_R}^{y} \frac{\lambda_R+B(s)}{g(s)} \d s \right) \int_{0}^{y} p \left(\frac{z}{y}\right) (1+z^k) \d z \d y
			\\ &\geq 1 - \int_{x_R}^{x} \frac{B(y)}{g(y)} \exp \left(-\int_{x_R}^{y} \frac{\lambda_R+B(s)}{g(s)} \d s \right) \left( p_0 + p_k y^k\right) \d y.
		\end{align*}
		Since $\phi_R(R)=0$ we deduce that for all $R>R_0$, 
		\begin{equation}\label{eq:lowerbound} \int_{x_R}^{R} \frac{B(y)}{g(y)} \exp \left(-\int_{x_R}^{y} \frac{\lambda_R+B(s)}{g(s)} \d s \right) \left( p_0 + p_k y^k\right) \d y\geq 1,\end{equation}
		and this enforces $\lambda_R$ to be bounded from above.
		Indeed, otherwise, there would exist a sequence $(R_n)_{n\geq0}$ and $x_\infty\in[0,A]$ such that
		\[R_n\to+\infty,\qquad \lambda_{R_n}\to+\infty,\qquad x_{R_n}\to x_\infty.\]
		But in that case, since
		\begin{align*}
			\mathds{1}_{[x_{R_n},R_n]}(y)\frac{B(y)}{g(y)}& \exp \left(-\int_{x_{R_n}}^{y} \frac{\lambda_{R_n}+B(s)}{g(s)} \d s \right) \left( p_0 + p_k y^k\right)\\
			&\leq\frac{B(y)}{g(y)} \exp \left(-\int_{A}^{y} \frac{B(s)}{g(s)} \d s \right) \left( p_0 + p_k y^k\right)
		\end{align*}
		and the latter function is integrable on $[0,+ \infty)$ (carry out
		an integration by parts and use~\eqref{eq:gB}), the dominated
		convergence theorem ensures that
		\[\int_{x_{R_n}}^{R_n} \frac{B(y)}{g(y)} \exp \left(-\int_{x_{R_n}}^{y} \frac{\lambda_{R_n}+B(s)}{g(s)} \d s \right) \left( p_0 + p_k y^k\right) \d y\to0,\]
		which contradicts~\eqref{eq:lowerbound}.	
	\end{proof}

	\begin{lem}
		\label{lem:boundonphiR}
		Under Hypotheses \ref{asmp:k1} and \ref{asmp:gB} with $p_0<+\infty$,
		$|\phi_R'(x)|$ is uniformly bounded on compact intervals for all
		$R> R_0$.
	\end{lem}
	
	\begin{proof}
		By the equation $\mathcal{L}^*_R\phi_R(x)=0$ and bounds on $\phi_R(x)$
		and $\lambda_R$ we obtain
		\begin{align*}
			| \phi_R'(x) |   &= \frac{\lambda_R \phi_R(x)}{g(x)} + \frac{B(x)}{g(x)}  \, \left | \phi_R(x) - \frac{1}{x} \int_{0}^{x}  p \left(\frac{y}{x} \right) \phi_R(y) \d y \right | 
			\\ &\leq \frac{\lambda_R }{g(x)} (1+x^k)+ \frac{B(x)}{g(x)} \, \left| 1+ x^k- \frac{1}{x} (1+x^k)\int_{0}^{x}  p \left(\frac{y}{x}  \right) \d y\right |
			\\ &\leq \frac{\lambda_R }{g(x)} (1+x^k)+ \frac{B(x)}{g(x)} \, \left| 1+ x^k- \frac{1}{x} (1+x^k)x p_0\right | 
			\\ &\leq \frac{\lambda_R }{g(x)} (1+x^k)+ \frac{B(x)}{g(x)} \, |1- p_0 |,
		\end{align*}
		which gives a bound on $\phi_R' (x)$ for all $R > R_0$, taking
		into account that $\lambda_R$ is uniformly bounded for all
		$R > R_0$ thanks to Lemma \ref{lem:boundonlambda}.
	\end{proof}
	
	\begin{proof}[Proof of Theorem \ref{thm:dualeigenfunction}]
		Lemmas \ref{lem:existencetruncated}, \ref{lem:bounonphi},
		\ref{lem:boundonlambda} and \ref{lem:boundonphiR} give the proof.
		Since there exists a solution to the truncated dual Perron
		eigenproblem \eqref{eq:trdualeigenfunction1} for any $R >0$ by Lemma~\ref{lem:existencetruncated}, it only remains to prove that the
		terms are bounded in order to pass to the limit as $R \to + \infty$.
		We provide the bounds on $\phi_R$, $\lambda_R$ and $\phi_R'$ by Lemmas
		\ref{lem:bounonphi}, \ref{lem:boundonlambda}, \ref{lem:boundonphiR}
		respectively.
		These bounds ensure that we can extract a subsequence of
		$(\lambda_R)$ which converges to $\lambda>0$ and a subsequence of
		$(\phi_R)$ which converges locally uniformly to a limit $\phi$ which
		satisfies $0 < \phi(x)\leq 1+x^k.$ Clearly $(\lambda,\phi)$ is the
		solution to the dual Perron eigenproblem
		\eqref{eq:dualeigenfunction}, and $\phi\not\equiv0$ since
		$\sup_{x \in [0,A]} \phi (x)= 1.$
		Similarly, the proof of the positivity or nullity of $\phi(0)$ is a
		direct consequence of~\cite[Theorem 1.10]{BCG13}.
	\end{proof}
	
	\section{Harris's Theorem}
	\label{sec:harris}
	
	In this section, we state Harris's theorem based on \cite{H16} and
	\cite{HM11}. The original idea comes from the study of discrete-time
	Markov processes and dates back to Doeblin and \citep{H56} where
	conditions of existence and uniqueness of having an equilibrium (or an
	\emph{invariant measure}) for a Markov process are investigated. It is
	a probabilistic method which relies on both a minorisation property
	and a drift condition (also called Foster-Lyapunov condition), which
	we describe below.
	
	We use Harris's theorem applied to continuous-time Markov processes in
	order to show that solutions to rescaled growth-fragmentation
	equation \eqref{eq:gfscaledgeneral}, under suitable assumptions,
	converge towards a universal profile at an exponential rate.
	
	\
	We assume that $\Omega$ is a Polish space and $(\Omega, \Sigma)$ is a
	measurable space together with its Borel $\sigma$-algebra $\Sigma$, so
	that $\Omega $ endowed with any probability measure is a Lebesgue
	space. Moreover we denote the space of finite measures on $\Omega$ by
	$\mathcal{M}(\Omega)$ and the space of probability measures on
	$\Omega$ by $\mathcal{P}(\Omega)$.
	
	A discrete-time Markov process $x$ is defined through a
	\emph{transition probability function}. A linear, measurable function
	$S \: \Omega \times \Sigma \mapsto \mathcal{P}(\Omega)$ is a
	transition probability function if $S(x, \cdot)$ is a probability
	measure for every $x$ and $x \mapsto S(\cdot, A)$ is a measurable
	function for every $A \in \Sigma$. By using the transition probability
	function we can define the associated Markov operator $\mathcal{S}$
	acting on the space of signed measures on $\Omega$ and its adjoint
	$\mathcal{S}^*$ acting on the space of bounded measurable functions
	$\varphi : \Omega \mapsto [0, +\infty)$ in the following way:
	\begin{equation*}
		(\mathcal{S} \mu) (A)
		= \int_{\Omega} S(x,A) \mu (\d x),
		\qquad (\mathcal{S}^* \varphi) (x) = \int_{\Omega} \varphi(y) S(x, \d y).  
	\end{equation*}
	On the other hand, a continuous-time Markov process is no longer
	described by a single transition function, but by a family of
	transition probability functions $S_t$ defined for each time $t
	>0$, with the property that the associated operators $\mathcal{S}_t$
	satisfy
	\begin{itemize}
		\item the semigroup property:
		$\mathcal{S}_{s+t} = \mathcal{S}_s \mathcal{S}_t$,
		\item and $\mathcal{S}_0$ is the identity, or equivalently,
		$S_0(x, \cdot) = \delta_x$ for all $x \in \Omega$.
	\end{itemize}
	We notice that $\mathcal{S}_{t}$ is \emph{linear, mass preserving} and
	\emph{positivity preserving}.  An \emph{invariant measure} of a
	continuous-time Markov process $(\mathcal{S}_t)_{t \geq 0}$ is a
	probability measure $\mu$ on $\Omega$ such that $\mathcal{S}_t \mu = \mu $
	for every $t \geq 0$, and it is the main concept we need to
	investigate when studying the asymptotic behaviour of a Markov process.
	
	\
	
	Let us state Doeblin's and Harris's theorems along with some
	hypotheses. We always assume $(\mathcal{S}_t)_{t \geq 0}$ is a
	continuous-time Markov semigroup. For their proofs we refer to
	\cite{MT93} or \cite{HM11,H16}.
	
	\begin{hyp}[Doeblin's condition]
		\label{hyp:Doeblin}
		There exists a time $t_0 >0$, a probability distribution $\nu$ and a
		constant $\alpha \in (0,1)$ such that for any initial condition
		$x_0$ in the domain we have:
		\[ \mathcal{S}_{t_0} \delta_{x_0} \geq \alpha \nu .\]
	\end{hyp}
	Using this we prove the following theorem:
	
	\begin{thm}[Doeblin's Theorem] \label{thm:Doeblin}
		
		If we have a Markov semigroup $(\mathcal{S}_t)_{t \geq 0}$
		satisfying Doeblin's condition (Hypothesis \ref{hyp:Doeblin}) then
		for any two finite measures $\mu_1$ and $\mu_2$ and any integer
		$n \geq 0$ we have that
		\begin{equation*}
			\label{eqn:Doeblin1}
			\left\| \mathcal{S}^{n}_{t_0} (\mu_1 -  \mu_2)\right\|_{\mathrm{TV}} \leq (1-\alpha) ^n\left\| \mu_1 - \mu_2 \right\|_{\mathrm{TV}}.
		\end{equation*}
		As a consequence, the semigroup has a unique invariant probability
		measure $\mu_*$, and for all probability measures $\mu$:
		\begin{equation*}
			\label{eqn:Doeblin2}
			\left\| \mathcal{S}_{t} ( \mu - \mu_*) \right\|_{\mathrm{TV}}  \leq C e^{-\rho t} \left\| \mu - \mu_*\right\|_{\mathrm{TV}}, \quad t \geq 0,
		\end{equation*}
		where 
		\[ C := \frac{1}{1-\alpha} >0,
		\qquad   \rho := \frac{-\log(1-\alpha)}{t_0} >0 .\]
	\end{thm}
	
	\
	
	Harris's theorem is an extension of Doeblin's theorem to situations in
	which one cannot prove a uniform minorisation condition as in
	Hypothesis \ref{hyp:Doeblin}. This is often the case when the state
	space is unbounded. Instead, we use Doeblin's condition only in a
	given region, and then show that the stochastic process will return to
	that region often enough. This is established by finding a so-called
	Lyapunov, or Foster-Lyapunov function. Both conditions then imply the
	existence of a spectral gap in a weighted total variation
	norm. Precisely, we need the following two hypotheses to be satisfied:
	
	\begin{hyp}[Foster-Lyapunov condition]
		\label{hyp:Lyapunov}
		There exist $\gamma \in (0,1)$, $K \geq 0$, some time $t_0 >0$
		and a measurable function $V : [0, + \infty) \mapsto [1, + \infty)$
		such that
		\begin{equation} \label{eq:hyp2}
			(\mathcal{S}_{t_0}^* V)(x) \leq \gamma V(x) + K, 
		\end{equation}
		for all $x$.
	\end{hyp}
	
	\begin{rem}
		When our continuous continuous-time Markov process is obtained by
		solving a particular PDE we often denote
		\begin{equation*}
			(\mathcal{S}_t m_0)(x) \equiv m(t,x),
		\end{equation*}
		where $m$ is the solution to the PDE with initial condition
		$m_0$. Then the previous condition is equivalent to
		\begin{equation*} \label{hyp:conf1}
			\int_\Omega m(t_0,x )V (x) \d x  \leq \gamma  \int_\Omega m_0(x) V(x) \d x + K ,
		\end{equation*}
		to be satisfied for all $m_0 \in \mathcal{P} (\Omega)$. One can
		verify this by proving the inequality
		\begin{equation*} \label{hyp:conf2}
			\frac{\d}{\d t} \int_\Omega m(t,x) V (x) \d x
			\leq -\lambda  \int_{0}^{+\infty} m(t,x) V(x) \d x + D
		\end{equation*}
		for some positive constants $D$ and $\lambda$, which then implies
		\eqref{eq:hyp2} with $\gamma = e^{-\lambda t_0}$ and
		$K = D/\lambda$.
	\end{rem}
	
	\
	
	The next hypothesis is a minorisation condition like Hypothesis
	\ref{hyp:Doeblin}, but only on a sufficiently large region:
	
	\begin{hyp}[Small set condition]
		\label{hyp:localDoeblin}
		There exist a probability measure $\nu$, a constant
		$\alpha \in (0,1)$ and some time $t_0 >0$ such that
		\[ \mathcal{S}_{t_0} \delta_{x_0} \geq \alpha \nu, \] for all
		$x_0 \in \mathcal{C}$, where
		\[ \mathcal{C} = \left \{ x : V(x) \leq R \right \} \] for some
		$R > 2K / (1-\gamma) $ where $K, \gamma$ are as in \emph{Hypothesis
			\ref{hyp:Lyapunov}}.
	\end{hyp}
	
	Finally we state \emph{Harris's theorem} under these hypotheses:
	
	\begin{thm}[Harris's Theorem] \label{thm:Harris}
		If we have a Markov semigroup $(\mathcal{S}_t)_{t \geq 0}$ satisfying Hypotheses \ref{hyp:Lyapunov} and \ref{hyp:localDoeblin} then there exist $\beta >0$ and $\bar{\alpha} \in (0,1)$ such that 
		\begin{equation*}
			\left\| \mathcal{S}_{t_0} \mu_1 - \mathcal{S}_{t_0} \mu_2\right\|_{V,\beta} \leq \bar{\alpha} \left\| \mu_1 - \mu_2 \right\|_{V,\beta}.
		\end{equation*} for all nonnegative measure $\int \mu_1 = \int \mu_2$, where the norm $\| \cdot \|_{V,\beta}$ is defined by
		\[\left\| \mu_1 - \mu_2 \right\|_{V,\beta} := \int (1+\beta V(x)) |\mu_1 - \mu_2 |\d x.\]
		Moreover, the semigroup has a unique invariant probability measure $\mu_*$ and there exist $C>0$ and $\rho >0$ (depending only on $t_0, \alpha, \gamma, K, R$ and $\beta$) such that 
		\begin{equation*}
			\left\|\mathcal{S}_t (\mu -\mu_*)\right\|_{V,\beta} \leq C e^{-\rho t} \left\|\mu -\mu_*\right\|_{V,\beta} \text{  for all } t \geq 0.
		\end{equation*} 
		Explicitly if we set $ \gamma_0 \in \left[\gamma + 2K /R, 1\right)$ for any $\alpha_0 \in (0, \alpha)$ we can chose $ \beta = \alpha_0/K$ and 
		$\bar{ \alpha } = \max \left \{ 1-\alpha + \alpha_0, (2+ R\beta \gamma_0) / (2+ R\beta) \right \}$. Then we have $C = 1/ \bar{\alpha}$ and $\rho = -(\log \bar{\alpha}) / t_0$.
	\end{thm}
	Proofs of Theorems \ref{thm:Doeblin} and \ref{thm:Harris} can be found
	for example in \cite{ H16, HM11, MT93, S13}.

	\section{Foster-Lyapunov condition}
	
	\label{sec:lyapunov}
	
	In this section we prove that Hypothesis \ref{hyp:Lyapunov} is
	verified for the semigroup generated by rescaled
	growth-fragmentation equation \eqref{eq:gfscaledgeneral}, when we
	consider the evolution of $f(t,x) := \phi(x) m(t,x)$.
	We divide the proof of Hypothesis \ref{hyp:Lyapunov} into
	three cases which require slightly different calculations.
	
	\subsection{Linear growth rate}
	\label{subseq:lingr}
	
	First we treat the linear growth case $g(x)=x$ with a constant
	fragmentation kernel. (As remarked before, we do not consider the
	mitosis kernel when $g(x)=x$ since there is no spectral gap in that
	case). In this case the Perron eigenvalue and the corresponding dual
	eigenfunction are known ($\lambda =1$ and $\phi(x)=x$), and the rescaled
	growth-fragmentation equation is given by
	\begin{equation}
		\label{g=xp=2}
		\frac{\p }{\p t}m(t,x) + \frac{\p }{\p x} (x m(t,x))
		= 2\int_{x}^{+\infty}  \frac{B(y)}{y} m(t,y) \d y - (B(x) +1) m(t,x),
	\end{equation}
	coupled with the usual initial and boundary conditions.
	
	\begin{lem}
		\label{lyapunovBlin}
		We consider Equation \eqref{g=xp=2} under Hypotheses \ref{asmp:k1}, \ref{asmp:gB} with a growth rate $g(x) = x$ and the constant fragment distribution $p(z) = 2$ for $z \in (0,1]$. Then the following holds true for any $K>1>k>-1$, for some $C_1, \bar{C}>0$, and any nonnegative measure solution $m = m(t,x)$:
		\begin{align}
			\label{eq:lyapunovBlin}
			\int_{0}^{+\infty} V(x) f(t,x) \d x   \leq e^{-C_1t} \int_{0}^{+\infty} V(x) f_0(x) \d x + \bar{C } \int_{0}^{+\infty} f_0(x) \d x
		\end{align}
		for all $t \geq 0$, where $f(t,x):=xm(t,x)$, $f_0(x) = x m_0(x)$, $\|f_0\|_V < +\infty$ and $V(x) = x^{k-1} + x^{K-1}$.
	\end{lem}

	\begin{proof}
		Let $\varphi:\R\to[0,1]$ be a non-increasing $C^1$ function such that $\varphi(x)=1$ for $x\leq0$ and $\varphi(x)=0$ for $x\geq1$.
			For $\ell>0$ we define $\varphi_\ell(x)=\varphi(x-\ell)$.
			Starting from~\eqref{eq:gf_def} we have
			\begin{align*}
				\frac{\d}{\d t} &\int_{0}^{+\infty} \big(x^k + x^K\big) \varphi_\ell(x)m(t,x) \d x \\
				&= \int_{0}^{+\infty}\left(\big(kx^{k-1}+Kx^{K-1}\big)\varphi_\ell(x)+\big(x^k+x^K\big)\varphi'_\ell(x)\right)xm(t,x)\d x\\
				&\qquad+2\int_{0}^{+\infty}\frac{B(x)}{x}m(t,x)\int_0^x\big(y^k+y^K\big)\varphi_\ell(y)\d y\d x\\
				&\qquad\qquad-\int_{0}^{+\infty}\big(1+B(x)\big)\big(x^k+x^K\big)\varphi_\ell(x)m(t,x)\d x.
			\end{align*}
			Since $\varphi_\ell$ is non-increasing we get
			\begin{align*}
				\frac{\d}{\d t} &\int_{0}^{+\infty} \big(x^k + x^K\big) \varphi_\ell(x)m(t,x) \d x \\
				&\leq\int_{0}^{+\infty}\big(kx^{k-1}+Kx^{K-1}\big)\varphi_\ell(x)xm(t,x)\d x\\
				&\qquad+2\int_{0}^{+\infty}B(x)\left(\frac{x^{k-1}}{k+1}+\frac{x^{K-1}}{K+1}\right)\varphi_\ell(x)xm(t,x)\d x\\
				&\qquad\qquad-\int_{0}^{+\infty}\big(1+B(x)\big)\big(x^{k-1}+x^{K-1}\big)\varphi_\ell(x)xm(t,x)\d x\\
				&\leq  -\frac{1}{2}(1-k)\int_{0}^{+\infty}(x^{k-1} + x^{K-1})\varphi_\ell(x)  xm(t,x) \d x \\
				&\quad+ \int_{0}^{+\infty}  \left( c_1 B(x) x^{K-1}+ c_2x^{K-1} + c_3 B(x) x^{k-1} + c_4x^{k-1} \right)\varphi_\ell(x)  xm(t,x) \d x
		\end{align*}
		where 
		\[ -1< c_1  : = \frac{1-K}{1+K}< 0, \, \,\, c_2:=K - \frac{k+1}{2}>0, \, \,\,  c_3:= \frac{1-k}{1+k} > 0, \,\,\, c_4:=\frac{k-1}{2}  <0 . \]
		We define 
		\begin{equation}\label{Phi_lin}
			\Phi(x) : = c_1 B(x)x^{K-1} + c_2x^{K-1} +  c_3 B(x)x^{k-1}+ c_4 x^{k-1}.
		\end{equation}
		Due to Hypothesis \ref{asmp:gB}, the total fragmentation rate
		$B \: [0, + \infty) \to [0, + \infty) $ satisfies
		$B(x)  \to 0$ as $x \to 0$ and
		$B(x) \to +\infty$ as $x \to +\infty$. Hence in
		the latter expression the behaviour as $x \to +\infty$ is dominated
		by the first term; thus $\Phi(x)$ will approach $-\infty$.
		Similarly when $x \to 0$, the last term will dominate the behaviour
		of $\Phi$, which is negative as well. Since $B$ is continuous we can
		always bound ${\sup}_{x \geq 0} \, \Phi(x) \leq C_2 $ with
		some positive quantity $C_2 >0.$ Therefore by denoting
		$f(t,x) = xm(t,x)$ and $f_0(x)=x m_0(x)$ we obtain, since $\varphi_\ell\leq1$ and $\int f(t,x)\d x=\int f_0(x)\d x$,
			\begin{align*}
				\frac{\d}{\d t} \int_{0}^{+\infty}& (x^{k-1}+x^{K-1}) \varphi_\ell(x)f(t,x) \d x  \\
				& \leq -C_1 \int_{0}^{+\infty} (x^{k-1}+x^{K-1})\varphi_\ell(x) f(t,x) \d x + C_2  \int_{0}^{+\infty} f_0(x) \d x,
		\end{align*}
		where $C_1 = (1-k) /2 >0$. Then Grönwall's lemma implies
			\begin{align*}
				\int_{0}^{+\infty} V(x)\varphi_\ell(x) f(t,x) \d x &  \leq e^{-C_1t} \int_{0}^{+\infty} V(x)\varphi_\ell(x) f_0(x) \d x + \bar{C } \int_{0}^{+\infty} f_0(x) \d x
			\end{align*}
			with $\bar{C} = C_2/C_1$.
			Due to the monotone convergence theorem we deduce \eqref{eq:lyapunovBlin} by letting $\ell$ go to $+\infty$.
	\end{proof}
	
	\subsection{Sublinear growth rate close to $0$}
	\label{subseq:sublingr}
	
	In this section we assume that $\int_0^1 \frac 1g<+\infty$, which we
	sometimes refer to as the case of \emph{sublinear growth rate at
		$x=0$}.
	
	\begin{lem}
		\label{lem:lyapunov1} 
		We consider Equation
		\eqref{eq:gfscaledgeneral} under Hypotheses \ref{asmp:k1},
		\ref{asmp:gB}, and $\int_0^1 \frac 1g<+\infty$. We take $K>1+\xi$.
		Then the following holds true for $C_1 = \lambda$ (the first
		eigenvalue), some $C_2 > 0$, and any nonnegative measure solution $m = m(t,x)$:
		\begin{equation}  \label{lya1}
			\frac{\d}{\d t} \int_{0}^{+\infty} x^K m(t,x) \d x
			\leq -C_1 \int_{0}^{+\infty} x^K m(t,x) \d x
			+ C_2 \int_{0}^{+\infty}  \phi(x)m(t,x) \d x,
		\end{equation}
		for all $ t \geq 0$.
	\end{lem}
	
	\begin{proof}
		For the sake of conciseness and clarity, we skip the truncation procedure here. But the same method as for Lemma~\ref{lyapunovBlin} can be used to make the calculations rigorous by using the truncation function $\varphi_\ell$.
		We have
		\begin{align*}
			\frac{\d}{\d t} &\int_{0}^{+\infty} x^K m(t,x) \d x
			\\&= -\int_{0}^{+\infty}  x^K\frac{\p}{\p x} \left(g(x) m(t,x) \right) \d x - \int_{0}^{+\infty} x^K  (B(x) +\lambda) m(t,x) \d x
			\\&+\int_{0}^{+\infty} x^K\int_{x}^{+\infty} \frac{B(y)}{y} p\left(\frac{x}{y}\right) m(t,y) \d y \d x
			\\&=  -\lambda \int_{0}^{+\infty}  x^K m(t,x) \d x + \int_{0}^{+\infty} \left( (p_K-1)x^{K} B(x) + K x^{K -1}g(x) \right) m(t,x) \d x.
		\end{align*}
		We define \[\Phi(x):= (p_K-1)x^{K} B(x) + K x^{K -1}g(x)\] and
		notice that ${\sup}_{x \geq 0} \, \Phi(x) \leq C_2\phi(x)$
		for some $C_2 >0$ due to Hypothesis \ref{asmp:gB} concerning the
		behaviour of $xB(x)/ g(x)$ as $x \to +\infty$ and
		$x \to 0$, and the fact that $\phi(0)>0$ since
		$\int_0^1 \frac{1}{g}<+\infty$ which is a result of Theorem
		\ref{thm:dualeigenfunction}.
	\end{proof}
	
	We now give a translation of this lemma in terms of $f = \phi m$, since this is needed in order to apply Harris's theorem to the evolution of $f$:
	
	\begin{cor}
		\label{cor:lyapunov1/gint}
		We consider Equation
		\eqref{eq:gfscaledgeneral} under Hypotheses \ref{asmp:k1},
		\ref{asmp:gB}, and $\int_0^1 \frac 1g<+\infty$. For
		$V(x) = 1+ \frac{x^K}{\phi(x)}$ where $K > 1 + \xi$ and
		$f(t,x):=\phi(x)m(t,x)$ with $ f_0(x)=\phi(x) m_0(x), \|f_0\|_V < +\infty$,  there exist $C_1, \tilde{C} >0$ such that
		for all $t\geq0$
		\begin{equation}
			\label{eq:lyapunov1}
			\int_{0}^{+\infty} V(x) f(t,x) \d x   \leq e^{-C_1t} \int_{0}^{+\infty} V(x) f_0(x) \d x + \tilde{C} \int_{0}^{+\infty} f_0(x) \d x.
		\end{equation}
	\end{cor}
	
	\begin{proof} By adding $\phi(x)$ of both sides of \eqref{lya1} we obtain
		\begin{align*}
			\frac{\d}{\d t} &\int_{0}^{+\infty} x^K m(t,x) \d x = \frac{\d}{\d t} \int_{0}^{+\infty} (x^K+ \phi(x)) m(t,x) \d x \\
			&\leq -C_1 \int_{0}^{+\infty} (x^K+\phi(x)) m(t,x) \d x
			+ (C_1+C_2) \int_{0}^{+\infty} \phi(x) m(t,x) \d x.
		\end{align*}
		Therefore, we have for $f(t,x) = \phi(x) m(t,x)$;
		\begin{multline*} 
			\frac{\d}{\d t} \int_{0}^{+\infty} \left(1+\frac{x^K}{\phi(x)}\right) f(t,x) \d x 
			\\ \leq -C_1 \int_{0}^{+\infty} \left(1+\frac{x^K}{\phi(x)}\right) f(t,x) \d x + (C_1 + C_2) \int_{0}^{+\infty} f_0(x) \d x,
		\end{multline*}  since $\int f(t,x)dx=\int f_0(x)dx$.
		Grönwall's lemma then implies \eqref{eq:lyapunov1} with $\tilde{C} = 1+C_2/C_1$.
	\end{proof}

	\subsection{Superlinear growth rate close to $0$}
	\label{subseq:superlingr}
	
	Now we assume that $\int_0^1\frac 1g=+\infty$, which implies linear or superlinear behaviour for the growth rate $x$ close $0$. This, of course, includes the case $g(x)=x$ from Section \ref{subseq:lingr}, but the general result we obtain now is slightly more restrictive. In the case
	of exact linear growth, Lemma \ref{lyapunovBlin} is slightly more
	precise.
	\begin{lem}
		\label{lem:lyapunov2} 
		We consider Equation \eqref{eq:gfscaledgeneral} under Hypotheses
		\ref{asmp:k1}, \ref{asmp:gB}, and $\int_0^1\frac 1g=+\infty$. We
		take $k<0$ and $K>1 + \xi$. Then the following holds true for any
		nonnegative measure solution $m=m(t,x)$:
		\begin{multline*}
			\frac{\d}{\d t} \int_{0}^{+\infty} (x^k+x^K) m(t,x) \d x
			\leq -C_1 \int_{0}^{+\infty} (x^k+x^K) m(t,x) \d x
			+ C_2 \int_{0}^{+\infty}  \phi(x)m(t,x) \d x,
		\end{multline*}
		for all $t \geq 0$, where $C_1 = \lambda >0$ and $C_2>0$ is some
		constant independent of the solution $m$.
	\end{lem}
	
	\begin{proof}
		Here again we skip the truncation procedure and refer to the proof of Lemma~\ref{lyapunovBlin} for the method which allows making the calculations rigorous.
		We have
		\begin{align*}
			&\frac{\d}{\d t} \int_{0}^{+\infty} (x^k+x^K) m(t,x) \d x
			\\&= -\int_{0}^{+\infty}  (x^k+x^K)\frac{\p}{\p x} \left(g(x) m(t,x) \right) \d x - \int_{0}^{+\infty} (x^k+x^K)  (B(x) +\lambda) m(t,x) \d x
			\\& \quad + \int_{0}^{+\infty}  \frac{B(y)}{y} m(t,y) \int_0^1 (y^kz^k+y^Kz^K) p\left(z\right) y \d z\d y
			\\&=  -\lambda \int_{0}^{+\infty} (x^k+x^K) m(t,x) \d x 
			\\&\quad+ \int_{0}^{+\infty} \left( (p_k-1)x^kB(x) + (p_K-1)x^{K} B(x) + k x^{k-1}g(x) + K x^{K -1}g(x) \right) m(t,x) \d x
		\end{align*}
		Similarly to previous proofs, we define
		\[\Phi(x):= (p_k-1)x^kB(x) + (p_K-1)x^{K} B(x) + k x^{k-1}g(x)
		+ K x^{K -1}g(x)\] and notice that
		${\sup}_{x > 0} \, \Phi(x) \leq C_2\phi(x)$ for some
		$C_2 >0$ due to Hypothesis \ref{asmp:gB} concerning the
		behaviour of $xB(x)/g(x)$ as $x \to +\infty$ and
		$x \to 0$, and the fact that $p_K-1<0$ and $k<0$.
	\end{proof}
	
	\begin{cor}
		\label{cor:lyapunov1/gnonint}
		We consider Equation \eqref{eq:gfscaledgeneral} under Hypotheses
		\ref{asmp:k1}, \ref{asmp:gB} and $\int_0^1\frac 1g=+\infty$. For $V(x) = \frac{x^k+x^K}{\phi(x)}$
		with $k<0$, $K > 1 + \xi$, and $f(t,x):=\phi(x)m(t,x)$ with $f_0(x)= \phi (x)m_0(x)$, $\|f_0\|_V < +\infty$, there exist
		$C_1, \tilde{C} >0$ such that for all $t\geq0$:
		\begin{equation}
			\label{eq:lyapunov2}
			\int_{0}^{+\infty} V(x) f(t,x) \d x   \leq e^{-C_1t} \int_{0}^{+\infty} V(x) f_0(x) \d x + \tilde{C} \int_{0}^{+\infty} f_0(x) \d x.
		\end{equation}
	\end{cor}
	
	\begin{proof}
		The inequality in Lemma \ref{lem:lyapunov2} yields, for
		$f(t,x) := \phi(x) m(t,x)$,
		\begin{multline*}
			\frac{\d}{\d t} \int_{0}^{+\infty} \frac{x^k+x^K}{\phi(x)} f(t,x) \d x 
			\leq -C_1 \int_{0}^{+\infty} \frac{x^k+x^K}{\phi(x)} f(t,x) \d x + (C_1 + C_2) \int_{0}^{+\infty} f_0(x) \d x,
		\end{multline*} since $\int f(t,x)dx=\int f_0(x)dx$. 
		
		Then Grönwall's
		lemma implies \eqref{eq:lyapunov2} with
		$\tilde{C} = 1+C_2/C_1$.
	\end{proof}
	
	\section{Minorisation condition}
	\label{sec:LowerBounds}
	
	In this section, we show that Hypothesis \ref{hyp:localDoeblin} is
	verified for the semigroup generated by rescaled
	growth-fragmentation equation \eqref{eq:gfscaledgeneral}. We give the
	proof in two parts where the uniform fragment distribution and the
	equal mitosis are considered separately.
	
	\medskip
	
	We start by recalling some known results on the solution of the
	transport part of Equation \eqref{eq:gfscaledgeneral}. Consider the equation
	\begin{equation} \label{eq:growth1}
		\begin{aligned}
			\frac{\p }{\p t}m(t,x) + \frac{\p }{\p x} (g(x) m(t,x)) &= -c(x) m(t,x), \qquad &&t,x > 0,\\
			m(t,0) &= 0, \qquad &&t > 0,\\
			m(0,x) &=  n_0(x), \qquad  &&x > 0,
		\end{aligned}
	\end{equation} which is the same as Equation \eqref{eq:gfscaledgeneral} without the positive part of the fragmentation operator. 
	We remark that Hypothesis \ref{asmp:gB} ensures that the characteristic ordinary differential equation
	\begin{align}
		\begin{split}
			\label{eq:char-ode}
			\ddt X_t (x_0) &= g(X_t(x_0)), \\
			X_0(x_0) &= x_0,
		\end{split}
	\end{align}
	has a unique solution, defined for $t \in [0,+\infty)$, for any
	initial condition $x_0 > 0$.  In fact, it is defined in some interval
	$(t_*(x_0), +\infty)$, for some $t_*(x_0) < 0$. The solution can be
	explicitly given in terms of $H^{-1}$, where
	\begin{equation*}
		H(x) := \int_1^x \frac{1}{g(y)} \d y, \qquad x \geq 0.
	\end{equation*}
	We notice that $H$ is strictly increasing with
	$H_0 := H(0) = \underset{x \to 0}{\lim} \, H(x) < 0$ and
	$\underset{x \to +\infty}{\lim} \, H(x) = +\infty$ (since $g$ grows
	sublinearly as $x \to +\infty$), so that it is invertible as a map
	from $(0,+\infty)$ to $(H_0, +\infty)$. (We allow $H_0 = -\infty$ if
	$1/g$ is not integrable close to $x=0$.) It can easily be checked that
	\begin{equation*}
		\label{eq:char-solution}
		X_t(x_0) = H^{-1} (t + H(x_0))
		\qquad \text{for $x_0 > 0$ and $t > H_0 - H(x_0)$}, 
	\end{equation*}
	so that that the maximal time interval where the solution of
	\eqref{eq:char-ode} is defined is precisely as $(H_0-H(x_0),
	+\infty)$. Since it will be convenient later, we define
	\begin{equation*}
		X_t(0) := \lim_{x_0 \to 0} X_t(x_0) =
		\begin{cases}
			0 &\qquad \text{if $H_0 = -\infty$,}
			\\
			H^{-1}(t + H_0) &\qquad \text{if $H_0 \in (-\infty, 0)$.}
		\end{cases}
	\end{equation*}
	This reflects the fact that the characteristics take a very long time
	to escape from $0$ when $1/g$ is not integrable close to $0$; while
	they escape in finite time if $1/g$ is integrable close to $0$.
	For each $t \geq 0$, we have thus defined the \emph{flow map}
	$X_t \: (0,+\infty) \to (X_t(0),+\infty)$, which is strictly
	increasing. For negative times, we may consider
	$X_{-t} \: (X_t(0),+\infty) \to (0,+\infty)$ (where $t >
	0$). Of course, $X_{-t} = (X_t)^{-1}$.
	
	If $n_0$ is a nonnegative measure, it is well known that the unique
	measure solution to Equation \eqref{eq:growth1} is given by
	\begin{equation}
		\label{eq:growth-sol-form1}
		\begin{aligned}
			m(t, x) &= X_t \# n_0(x)
			\exp{\left( -\int_0^t c(X_{-\tau}(x)) \d \tau \right)},
			\qquad &&t \geq 0,\ x > X_t(0),
			\\
			m(t,x) &= 0, \qquad &&t \geq 0,\ x \leq X_t(0),
		\end{aligned}
	\end{equation}
	where we abuse notation by evaluating the measures $m(t,\cdot)$ and
	$X_t \# n_0$ at a point $x > 0$. For a Borel measurable map
	$X \: (0,+\infty) \to (0,+\infty)$, the expression $X \# n_0$ denotes
	the \emph{transport}, or \emph{push forward}, of the measure $n_0$ by
	the map $X$, defined by duality through
	\begin{equation*}
		\int_0^\infty \varphi(x) X \# n_0(x) \d x
		:= \int_0^\infty \varphi(X(y)) n_0(y) \d y
	\end{equation*}
	for all continuous, compactly supported
	$\varphi \: (0,+\infty) \to \R$.  We use the notation $\mathcal{T}_t$
	for this flow map:
	\begin{equation} \label{T_t}
		\mathcal{T}_t n_0(x) := X_t \# n_0(x), \qquad \text{for all } t \geq 0,
	\end{equation}
	so $\mathcal{T}_t$ is the semigroup associated to transport equation
	\eqref{eq:growth1}.
	
	If additionally $n_0$ is a function
	and $X$ has a left inverse $X^{-1}\: (a,b) \to (0,+\infty)$, one has
	\begin{equation*}
		X \# n_0(x) =
		\begin{cases}
			n_0(X^{-1}(x))\, \Big|\ddx (X^{-1})(x))\Big|
			\quad & \text{if $x \in (a,b)$,}
			\\
			0
			\quad & \text{otherwise.}
		\end{cases}
	\end{equation*}
	Using this for the solution to \eqref{eq:growth1}, if $n_0$ is a
	function we may write $m$ in the equivalent form
	\begin{equation}
		\label{eq:growth-sol-form2}
		m(t, x) = n_0(X_{-t}(x)) \ddx X_{-t}(x)
		\exp{\left( -\int_0^t c(X_{-\tau}(x)) \d \tau \right)}
	\end{equation}
	when $t \geq 0 $ and $x > X_t(0)$, and $m(t,x) = 0$
	otherwise. Using that $Y_t(x) := \ddx X_t(x)$ satisfies
	$\ddt Y_t(x) = g'(X_t(x)) Y_t(x)$, we note for later that
	\begin{equation}
		\label{eq:growth-sol-form3}
		\ddx X_{-t}(x) = \exp{\left( -\int_0^t g'(X_{-\tau}(x))\d \tau \right)},
		\qquad t \geq 0,\ x > X_t(0).
	\end{equation}

	\subsection{Uniform fragment distribution}
	
	Let us consider the case of uniform fragment distribution $p(z)=2$,
	corresponding to the fragmentation kernel of the form $\kappa (x,y) = \frac{2}{x} B(x) \1_{\{0 \leq x \leq y\}}$. The growth-fragmentation equation in this case is widely studied and depending on some assumptions made on growth and total division rates,
	existence (in some cases exact values) of eigenelements are known. The
	rescaled growth-fragmentation equation in this case becomes
	\begin{equation}
		\label{eq:gfscaledp=2}
		\begin{aligned}
			\frac{\p }{\p t}m + \frac{\p }{\p x} (g(x) m) &= 2\int_{x}^{+\infty}  \frac{B(y)}{y} m(t,y) \d y
			- (B(x) + \lambda ) m,  \quad &&t,x \geq 0,
			\\ m(t,0) &= 0,  \quad &&t > 0,
			\\ m(0,x) &= n_0(x),  \quad &&x > 0,
		\end{aligned}
	\end{equation}
	where $m = m(t,x)$ whenever variables are not explicitly written. If
	we consider a linear growth $g(x) = g_0x$ and a power like total
	division $ B(x) = b_0x^{\gamma}$ with $\gamma >0$, and $g_0, b_0 >0$,
	the Perron eigenvalue and the corresponding dual eigenfunction are
	given by 
	\begin{align*}
		\lambda = g_0 \quad \text{ and } \quad \phi(x) = \frac{x}{\int y N(y)}.
	\end{align*}
	In this case, eigenelements can be computed explicitly (see for
	example \cite{DG09}):
	\begin{equation*}
		\lambda = g_0, \qquad 
		N(x) = \left(\frac{b_0}{\gamma g_0}\right)^{1/\gamma}
		\frac{\gamma}{\Gamma \big(\frac{1}{\gamma}\big)}
		\exp \left(- \frac{1}{\gamma}\frac{b_0}{g_0} x^\gamma\right),
		\qquad
		\phi(x) = \left(\frac{b_0}{\gamma g_0}\right)^{1/\gamma}
		\frac{ \Gamma \big(\frac{1}{\gamma}\big)}
		{\Gamma \big(\frac{2}{\gamma}\big)} x.
	\end{equation*}
	Moreover, in \cite{BCG13}, the authors give the asymptotics of the
	profile $N$ and accurate bounds on the dual eigenfunction $\phi$ in a
	more general form of the growth-fragmentation equation where growth
	and total division rates behave like a power law for large and small
	$x$.
	
	\begin{lem}[Lower bound for the uniform fragment distribution]
		\label{lem:doeblinuniformfrag}
		Assume Hypotheses \ref{asmp:k1} and \ref{asmp:gB} hold true with a constant
		distribution of fragments $p(z) = 2$ for $z \in (0,1]$. Let
		$(\mathcal{S}_t)_{t \geq 0}$ be the linear semigroup associated to
		Equation \eqref{eq:gfscaledp=2}.  For all $0 < \eta < \theta$ given,
		there exists $t_0 >0$ such that for all $t > t_0$ and
		$x_0 \in (\eta,\theta]$ it holds that
		\begin{equation*}
			\mathcal{S}_{t}  \delta_{x_0}(x) \geq C(\eta, \theta, t)\qquad \text{for all $x \in I_t$},
		\end{equation*}
		where $I_t$ is an open interval which depends on $\eta$,
		the time $t$, and for some quantity $C = C(\eta, \theta, t)$
		depending only on $\eta$, $\theta$ and $t$. If in addition we assume
		that
		\begin{equation*}
			\int_0^1 \frac{1}{g(x)} \d x < +\infty,
		\end{equation*}
		then the above result also holds when taking $\eta = 0$.
	\end{lem}
	
	\begin{proof}
		Recall that $(\mathcal{T}_t)_{t \geq 0}$ the semigroup associated to the transport equation
		\begin{equation*} 
			\frac{\p }{\p t}m(t,x) +  \frac{\p }{\p x}  (g(x)m(t,x)) + c(x) m(t,x) =0,
		\end{equation*} where $c(x) = B(x) + \lambda$. By Duhamel's formula we have
		\begin{equation*}
			\label{duhamel}
			\mathcal{S}_t n_0(x) = m(t,x) = \mathcal{T}_t n_0(x) + \int_{0}^{t} \mathcal{T}_{t-\tau} (\mathcal{A}(\tau, .))(x) \d \tau,
		\end{equation*}
		where
		$\mathcal{A}(t,x) := 2 \int_{x}^{+\infty} \frac{B(y)}{y} m(t,y)
		\d y$. Fix $0 \leq \eta < \theta$, and take any
		$x_0 \in (\eta, \theta]$. 
		
		If $n_0 = \delta_{x_0}$, a simple
		bound gives
		\begin{equation*}
			\mathcal{S}_t \delta_{x_0} \geq \mathcal{T}_t \delta_{x_0}
			= X_t \# \delta_{x_0} \exp{\left( -\int_0^t c(X_{t-\tau}(x_0)) \d \tau \right)},
		\end{equation*}
		where we have used the expression of $\mathcal{T}_t$ given in
		\eqref{eq:growth-sol-form1} and the fact that the support of
		$X_t \# \delta_{x_0}$ is the single point $\{X_t(x_0)\}$. By
		Hypothesis \ref{asmp:gB} (in particular since $B$ is
		continuous on $[0,X_t(\theta)]$), for some $C_1 = C_1(\theta, t)$
		which is increasing in $t$, we have
		\begin{equation*}
			c(x) = B(x) + \lambda \leq C_1
			\quad \text{for all $x \in (0, X_t(\theta)]$}.
		\end{equation*}
		We deduce that
		\begin{equation} \label{secondbound}
			\mathcal{S}_t \delta_{x_0} \geq X_t \# \delta_{x_0} e^{-C_1 t}
			= \delta_{X_t(x_0)} e^{-C_1 t}.
		\end{equation}
		Using this we obtain
		\begin{equation*}
			\mathcal{A}(t,x) \geq 2 e^{-C_1 t} \frac{B(X_t(x_0))}{X_t(x_0)}
			\quad \text{for all $t > 0$ and $x < X_t(x_0)$.}
		\end{equation*}
		We use that there is some $x_B > 0$ for which $B$ is bounded
		below by a positive quantity on any interval of the form
		$[x_B, R]$. There is some $t_B > 0$ such that for $t > t_B$ we
		have $X_t(x_0) > x_B$ for all $x_0 > \eta$ {(for this to hold,
			notice we may take $\eta = 0$ in the case that
			$\int_0^1 1/g < +\infty$, but we need
			$\eta > 0$ otherwise)}. Hence, for some
		$C_2 = C_2(\eta, \theta, t)$ which is decreasing in $t$, we obtain 
		\begin{equation*}
			\mathcal{A}(t,x) \geq
			C_2 e^{-C_1 t}
			\qquad \text{for all $t > t_B$ and $x < X_t(x_0)$.}
		\end{equation*}
		Take now $t > t_B$, which will stay fixed until the end of the
		proof. The previous bound shows that
		\begin{equation*}
			\mathcal{A}(\tau,x) \geq
			C_2(\eta,\theta, \tau) e^{-C_1(\theta, \tau) \tau}
			\geq
			C_2(\eta,\theta, t) e^{-C_1(\theta, t) \tau}
			=:
			\tilde{C_2} e^{-\tilde{C_1} \tau}
		\end{equation*}
		{for all $t > t_B$,\ $t_B < \tau < t$ and all
			$x < X_\tau(x_0)$.} As a consequence, using
		\eqref{eq:growth-sol-form2} and \eqref{eq:growth-sol-form3},
		\begin{equation*}
			\mathcal{T}_{t-\tau} \mathcal{A}(\tau, x)
			\geq \tilde{C_2} e^{-\tilde{C_1} \tau}
			\exp{\left( -\int_0^{t-\tau} c(X_{-s}(x)) \d s \right)}
			\exp{\left( -\int_0^{t-\tau} g'(X_{-s}(x))\d s \right)}
		\end{equation*}
		for all $t_B < \tau < t$ and $X_{t-\tau}(0) < x <
		X_t(x_0)$. Since $X_{-s}(x) \leq X_t(x_0)$ in this range, we
		can bound this by
		\begin{equation*}
			\mathcal{T}_{t-\tau} \mathcal{A}(\tau, x)
			\geq \tilde{C_2} e^{-2 \tilde{C_1} t}
			\exp{\left( -\int_0^{t-\tau} g'(X_{-s}(x))\d s \right)},
		\end{equation*}
		again for all $t_B < \tau < t$ and
		$X_{t-\tau}(0) < x < X_t(x_0)$. In order to find a lower bound
		for the last exponential we restrict to a smaller $x$
		interval. Since the bound holds for all $x$ with
		\begin{equation*}
			X_{t-\tau}(0) < x < X_t(x_0),
		\end{equation*}
		it holds in particular for all $x$ with
		\begin{equation}
			\label{eq:It-interval}
			X_{t-t_B}(\eta) < x < X_t(\eta).
		\end{equation}
		Again this is a point where we need to take $\eta > 0$ in the
		case $\int_0^1 1/g = +\infty$, since otherwise this gives an
		empty range of $x$. In the case $\int_0^1 1/g < +\infty$,
		$\eta = 0$ is allowed. In this range, the quantity $X_{-s}(x)$
		inside the exponential satisfies
		\begin{equation*}
			X_{\tau-t_B}(\eta) \leq X_{-s}(x) \leq X_t(\eta)
		\end{equation*}
		Choose $\delta > 0$ such that $t_B + \delta < t$. Then for all
		$x$ satisfying \eqref{eq:It-interval} and all $\tau \in (t_B +
		\delta, t)$ we have
		\begin{equation*}
			X_{\delta}(\eta) \leq X_{-s}(x) \leq X_t(\eta).
		\end{equation*}
		Using that $g'(X) \leq C_3$ for all
		$X \in \left [ X_{\delta}(\eta), X_t(\eta) \right ]$ we have
		\begin{equation*}
			\mathcal{T}_{t-\tau} \mathcal{A}(\tau, x)
			\geq \tilde{C_2} e^{-\tilde{C_1} \tau} e^{-C_3 (t-\tau)}
			\geq \tilde{C_2} e^{-C_4 t}
		\end{equation*}
		for all $x$ satisfying \eqref{eq:It-interval} and all
		$\tau \in (t_B + \delta, t)$. A final integration gives, for
		$x$ in the same interval,
		\begin{equation*}
			\int_{0}^{t} \mathcal{T}_{t-\tau} (\mathcal{A}(\tau, \cdot))(x) \d\tau
			\geq \tilde{C_2} e^{-C_4 t}  \int_{t_B+\delta}^{t} \d\tau
			= \tilde{C_2} e^{-C_4 t} (t-t_B-\delta).
		\end{equation*}
		Taking $t_0:=t_B$ gives the result.
	\end{proof}
	
	\subsection{Equal mitosis}
	
	We now consider the fragment distribution
	$p(z) = 2 \delta_{\frac{1}{2}}(z)$ which describes the process of equal mitosis, in which cells of size $x$ split into two equal daughter
	cells of size $x/2$. In Equation \eqref{eq:gfscaledgeneral}, we have then
	$\mathcal{A}(t,x) := 4 B(2x) m(t,2x)$ and the rescaled
	growth-fragmentation equation takes the form
	\begin{equation}
		\label{eq:gfscaledmitosis}
		\begin{aligned}
			\frac{\p }{\p t}m(t,x) + \frac{\p }{\p x} (g(x) m(t,x))
			&= 4B(2x) m(t,2x) - (B(x) + \lambda ) m(t,x),  &&t,x \geq 0,
			\\ m(t,0) &= 0, \qquad \qquad &&t > 0,
			\\ m(0,x) &= n_0(x), \qquad &&x > 0.
		\end{aligned}
	\end{equation}
	The case where $g$ and $B$ are constant was the subject of numerous
	works in the past, most notably~\cite{BCGMZ13,CMP10,HW89,MS16,PR05,vBALZ}.
	For $g(x)=1$ and $B(x)=1$, eigenelements are
	given by
	\[\lambda = 1, \qquad 
	N(x) = \sum_{n=0}^{+\infty} (-1)^n \alpha_n e^{-2^{n+1}x},
	\qquad
	\phi(x) \equiv 1.
	\]
	with $\alpha_n=\frac{2}{2^n-1}\alpha_{n-1}$ and $\alpha_0>0$ a suitable normalization constant,
	and the solution $m(t,x)$ converges exponentially fast to the universal profile $N(x)$, which vanishes as $x\to0$ and $x\to+\infty$.
	However, when a linear growth rate $g(x)= x$ is considered
	Equation \eqref{eq:gfscaledmitosis} exhibits oscillatory behaviour in
	the long time. This is because instead of a dominant real eigenvalue,
	there are nonzero imaginary eigenvalues, so that there exists a set of
	dominant eigenvalues. This type of periodic long time behaviour was
	first observed in \cite{DHT84} and then it was proved
	in \cite{GN88} by using the theory of positive semigroups combined
	with spectral analysis to obtain the convergence to a semigroup of
	rotations. Since the method relies on some compactness arguments, the
	authors considered the equation in a compact subset of $(0,+\infty)$.
	Recently in~\cite{GM19}, the authors proved the oscillatory
	behaviour in the framework of measure solutions for general division rates on
	$(0,+\infty)$. The proof relies on a general relative
	entropy argument combined with the use of Harris's theorem on discrete sub-problems.
	It provides an explicit rate of convergence in weighted total variation norm.
	Here we consider a sublinear growth rate and a
	more general division rate than those so far considered in the
	literature. We exclude of course the case $g(x)=x$, for which we know
	the lower bound (and the exponential convergence) does not hold.
	
	\medskip
	We first need a technical lemma which gives an expression
	for the time integration of a measure moving in time:
	
	\begin{lem}
		\label{lem:time-integral-of-dirac}
		Let $t > 0$ and $F \: [0,t] \to \R$ an injective, differentiable
		function. Then
		\begin{equation*} 
			\int_{0}^{t} \delta_{F(\tau)} (x) \d \tau =\left(F^{-1}\right)'(x)   \1_{\{F(0) \leq x \leq F(t)\}}.
		\end{equation*}
	\end{lem}
	
	\begin{proof}
		Integrating against a smooth test function $\varphi(x)$ we obtain
		\begin{align*}
			\begin{split}
				\int_{0}^{+\infty} \varphi(x)	\int_{0}^{t} \delta_{F(\tau)} (x) \d \tau \d x 
				&= 	\int_{0}^{t} 	\int_{0}^{+\infty} \varphi(x) \delta_{F(\tau)} (x) \d x \d \tau 
				\\&= 	\int_{0}^{t} \varphi (F(\tau)) \d \tau = \int_{F(0)}^{F(t)} \varphi(y) \left(F^{-1}\right)'(y) \d y.
			\end{split}
		\end{align*}
		by using the change of variable $y = F(\tau)$.
	\end{proof}
	
	The following result will ensure a certain \emph{sublinearity} of the
	characteristic flow $X_t$ which we will need later:
	\begin{lem}
		\label{lem:flow-sublinear}
		Assume that the growth rate $g\: (0,+\infty) \to (0,+\infty)$ is
		locally Lipschitz and satisfies
		\begin{equation*}
			\omega g(x) < g(\omega x)
			\qquad \text{for all $x > 0$ and $\omega \in (0,1)$.}
		\end{equation*}
		Then for any $t > 0$ the characteristic flow $X_t$ satisfies
		\begin{equation*}
			\omega X_t(x) < X_t (\omega x),
			\qquad \text{for all $x > 0$ and $\omega \in (0,1)$.}
		\end{equation*}
	\end{lem}
	
	\begin{proof}
		Call $h_1(t) := \omega X_t(x)$ and $h_2(t) := X_t(\omega x)$. The
		second one satisfies the ODE
		\begin{equation*}
			h_2'(t) = g(h_2(t)),
		\end{equation*}
		while the first one satisfies
		\begin{equation*}
			h_1'(t) = \omega g(X_t(x)) < g (\omega X_t(x)) = g(h_1(t)).
		\end{equation*}
		Since they have the same initial condition, this differential
		inequality implies $h_1(t) < h_2(t)$ for all $t > 0$.
	\end{proof}
	
	Our main lower bound for the mitosis case is the following:
	
	\begin{lem}[Lower bound for equal mitosis]
		\label{lem:doeblinmitosis}
		Assume Hypotheses \ref{asmp:k1}, \ref{asmp:gB}, \ref{asmp:gp} hold true
		with the mitosis kernel $p(z) = 2\delta_{\frac{1}{2}}(z).$ Let
		$(\mathcal{S}_t)_{t \geq 0}$ be the semigroup associated to Equation
		\eqref{eq:gfscaledmitosis}. For any $\theta > 0$ there exists
		$t_0 = t_0(\theta) >0$ such that for all $t > t_0$ and
		$x_0 \in (0, \theta]$ it holds that
		\[
		\mathcal{S}_{t} \delta_{x_0}(x) \geq C (t_0, \theta)
		\qquad \text{for all $x \in I_t$},
		\]
		where $I_t$ is an open interval which depends on time $t$, and for
		some quantity $C=C(t, \theta)$ depending only on $t$ and $\theta$.
	\end{lem}
	
	\begin{proof}
		Fix $\theta > 0$ and take any $x_0 \in (0,\theta]$. We follow the
		same strategy as in the proof of Lemma \ref{lem:doeblinuniformfrag}. Here the only different part is $\mathcal{A}(t,x)$. We consider the
		semigroup $(\mathcal{T}_t)_{t \geq 0}$ defined as in \eqref{T_t} and
		$(\mathcal{S}_t)_{t \geq 0}$ defined as the semigroup associated to
		\eqref{eq:gfscaledmitosis} with
		$\mathcal{A}(t,x) = 4 B(2x) m(t,2x)$.  Using \eqref{secondbound} we have
		\[
		\mathcal{T}_t \delta_{x_0} (2x)
		\geq
		X_t \# \delta_{x_0} (2x)  e^{-C_1t}
		=
		\frac{1}{2}\delta_{\frac{1}{2} X_t \left(x_0 \right)} (x)
		e^{-C_1t},
		\]
		for $C_1 = C_1(\theta, t)$, increasing in $t$.
		we obtain 
		\[
		\mathcal{A} (t,x) \geq
		2 e^{-C_1 t} B \left(  X_t \left(x_0\right) \right)
		\delta_{\frac{1}{2} X_t \left(x_0 \right)} (x) 
		\quad \text{for all } t >0.
		\]
		We know that there exists some $x_B > 0$ for which $B$ is bounded below
		by a positive quantity in each interval of the form $[x_B, R]$. Take
		$t_B > 0$ such that for $t > t_B$ we have
		$X_t \left(x_0 \right) > x_B$ for all $x_0 > 0$. Hence, for some
		$C_2 = C_2(\theta, t) > 0$, decreasing in $t$,
		\begin{equation*}
			\mathcal{A}(t,x) \geq  C_2 e^{-C_1 t}
			\delta_{\frac{1}{2} X_t \left(x_0 \right)} (x)
			\qquad \text{for all $t > t_B$.}
		\end{equation*}
		Fix now any $t > t_B$. For $t_B < \tau < t$ we have
		\begin{multline*}
			\mathcal{A}(\tau,x)
			\geq
			C_2(\theta, \tau) e^{-C_1(\theta, \tau) \tau} \delta_{\frac{1}{2} X_\tau \left(x_0 \right)} (x) \geq
			C_2(\theta, t) e^{-C_1(\theta, t) t} \delta_{\frac{1}{2} X_\tau \left(x_0 \right)} (x)
			=:
			\tilde{C_2} e^{-\tilde{C_1} t} \delta_{\frac{1}{2} X_\tau \left(x_0 \right)} (x).
		\end{multline*}
		Hence using  \eqref{eq:growth-sol-form1} we have
		\begin{align*}
			\begin{split}
				\mathcal{T}_{t-\tau} \mathcal{A}(\tau, x)
				&\geq
				\tilde{C_2} e^{-\tilde{C_1} t}
				\delta_{ X_{t-\tau} \left( \frac{1}{2} X_\tau (x_0) \right)} (x)
				\exp{\left( -\int_0^{t-\tau} c(X_{-s}(x))\d s \right)} 
				\\
				&\geq \tilde{C_2} e^{-2\tilde{C_1} t} \delta_{ X_{t-\tau} \left( \frac{1}{2} X_\tau (x_0) \right)} (x),
			\end{split}
		\end{align*}
		for all $\tau \in (t_B, t)$. Define $F( \tau ) : = X_{t-\tau} \left(
		\frac{1}{2} X_\tau (x_0) \right)$, and notice that it is a
		strictly decreasing function, since Lemma \ref{lem:flow-sublinear}
		ensures that for $\tau_1 < \tau_2$
		\begin{equation*}
			F(\tau_2)
			= X_{t-{\tau_2}} \left( \frac{1}{2} X_{\tau_2} (x_0) \right)
			< X_{t - \tau_2} X_{\tau_2 - \tau_1}
			\left( \frac{1}{2} X_{\tau_1} (x_0) \right)
			= F(\tau_1).
		\end{equation*}
		By Lemma \ref{lem:time-integral-of-dirac} we
		obtain
		\begin{align*}
			\begin{split}
				\int_{0}^{t} \mathcal{T}_{t-\tau} \mathcal{A}(\tau, x)  \d \tau
				&\geq
				\int_{t_B}^{t} \mathcal{T}_{t-\tau} \mathcal{A}(\tau, x)  \d \tau
				\geq
				\tilde{C_2} e^{-2 \tilde{C_1} t} \int_{t_B}^{t} \delta_{ X_{t-\tau} \left(
					\frac{1}{2} X_\tau (x_0) \right)} (x) \d \tau
				\\
				&\geq \tilde{C_2} e^{-2 \tilde{C_1} t} \left ( F(\tau )\right)' (x)\1_{\mathcal{I}_{x_0}} 
			\end{split}
		\end{align*}
		where we define
		\begin{equation*}
			\label{interval}
			\mathcal{I}_{x_0} :=
			\left[
			\frac12 X_t \left (x_0 \right ),
			\
			X_{t-t_B}\left(\frac{1}{2} X_{t_B}(x_0) \right)
			\right ].
		\end{equation*}
		Again by Lemma \ref{lem:flow-sublinear} we see that this interval is
		nonempty. Since we need a bound which is independent of $x_0$, we
		consider the intersection of all these intervals as $x_0$ moves in
		the interval $(0, \theta)$. That intersection is
		\begin{equation*}
			\mathcal{I}_{t} :=
			\left[
			\frac12 X_t \left (\theta \right ),
			\
			X_{t-t_B}\left(\frac{1}{2} X_{t_B}(0) \right)
			\right ].
		\end{equation*}
		Condition \eqref{eq:H-1-power} shows that this interval is
		nonempty for $t$ large enough, since
		\begin{equation*}
			\frac{X_t \left (\theta \right )}
			{X_{t-t_B}\left(\frac{1}{2} X_{t_B}(0) \right)}
			=
			\frac{H^{-1} (t + \theta)}
			{H^{-1}\left(t - t_B + \frac{1}{2} X_{t_B}(0) \right) }
			\to 1
			\qquad
			\text{as $t \to +\infty$.}
		\end{equation*}
		This gives the result.
	\end{proof}

	\section{Proof of the main result}
	\label{sec:proof-main}
	
	We conclude by giving the proof of Theorem \ref{thm:main}. It is a
	direct application of Harris's Theorem \ref{thm:Harris}. Hypotheses \ref{hyp:Lyapunov} and \ref{hyp:localDoeblin} need to be verified. We
	already verified Hypothesis \ref{hyp:Lyapunov} (Lyapunov condition) in
	Section \ref{sec:lyapunov} (see the corollary given in each case); in
	fact, we have proved that given any $t_0 > 0$ we can satisfy
	Hypothesis~\ref{hyp:Lyapunov} for any $t \geq t_0$, with constants
	$\gamma$, $K$ \emph{which are independent of $t$} (since we can always
	take $\gamma := e^{-C_1 t_0}$, $K := \tilde{C}$).
	
	Regarding Hypothesis~\ref{hyp:localDoeblin}, the lower bounds we
	obtained in Section \ref{sec:LowerBounds} are for $m(t,x)$ which is a
	solution to Equation \eqref{eq:gfscaledgeneral}. However we need to satisfy the
	minorisation condition for $f(t,x) = \phi(x) m(t,x)$ since the
	equation on $f$ conserves mass; thus the associated semigroup is
	Markovian, and we may apply Harris's theorem to it.
	The equation satisfied by $f$ is
	\begin{equation}
		\begin{aligned}
			\label{eq:gfconserv}
			\frac{\p }{\p t}f(t,x) + \phi(x)\frac{\p }{\p x}
			&\left( \frac{g(x)}{\phi(x)} f(t,x) \right)
			+ (B(x) + \lambda) f(t,x) 
			\\
			&= \phi(x) \int_{x}^{+ \infty}  \frac{B(y)}{y} p\left(
			\frac{x}{y}\right) f(t,y) \d y, \qquad &&t,x \geq 0,\\
			m(t,0) &= 0, \qquad \qquad &&t > 0,\\
			m(0,x) &=  n_0(x), \qquad  &&x > 0.
		\end{aligned}
	\end{equation}
	We define $(\mathcal{F}_t)_{t \geq 0}$ as the semigroup associated to Equation
	\eqref{eq:gfconserv}, or alternatively by the relationship
	\begin{equation*}
		\mathcal{F}_t (\phi n_0) := \phi \mathcal{S}_t n_0,
	\end{equation*}
	for any nonnegative measure $n_0$ such that $\phi n_0$ is a
	finite measure on $(0,+\infty)$.
	
	\begin{lem}[Minorisation condition for $f(t,x)$]
		\label{lem:doeblingeneral}
		We assume Hypotheses \ref{asmp:k1}, \ref{asmp:p}, \ref{asmp:gB} and
		\ref{asmp:gp} hold true. Let $(\mathcal{F}_t)_{t \geq 0}$ be the
		semigroup associated to Equation \eqref{eq:gfconserv}. For any
		$0 > \eta >\theta$ there exists $t_0 = t_0(\eta, \theta) >0$ such that for
		all $t > t_0$ and $x_0 \in [\eta, \theta]$ it holds that
		\[ \mathcal{F}_{t} \delta_{x_0}(x) \geq \breve{C} (\eta, \theta, t)
		\qquad \text{for all $x \in I_t$},
		\]
		where $I_t$ is an open interval which depends on time $t$, and for
		some quantity $\breve{C} = \breve{C}(\eta, \theta,t)$ depending only
		on $\eta$, $\theta$ and $t$. If in addition we assume that
		\begin{equation*}
			\int_0^1 \frac{1}{g(x)} \d x < +\infty,
		\end{equation*}
		then the above result also holds when taking $\eta = 0$.
	\end{lem}
	
	\begin{proof}
		Let $(\mathcal{S}_t)_{t \geq 0}$ and $(\mathcal{F}_t)_{t \geq 0}$ be
		the semigroups associated to Equations \eqref{eq:gfscaledgeneral} and
		\eqref{eq:gfconserv} respectively. Under the conditions of Lemma
		\ref{lem:doeblinuniformfrag} we have a lower bound for
		$\mathcal{S}_{t} \delta_{x_0}(x) \geq C (\eta, \theta, t)$. It
		immediately translates to a lower bound on $\mathcal{F}_t$ in all
		cases:
		\begin{enumerate}
			\item If $\int_0^1 \frac{1}{g(x)} \d x = +\infty$, we know from
			\cite{BCG13} that $\phi(x)$ is bounded
			in each interval of the form $(0, \theta]$ {(since it is
				continuous and tends to a positive constant at
				$x=0$)}.
			
			\item If $\int_0^1 \frac{1}{g(x)} \d x = +\infty$, then since
			$\phi(x)$ is continuous there exist constants
			$\hat{C}_1(\eta, \theta)$, $\hat{C}_2(\eta, \theta) > 0$ such that
			$\hat{C}_1 \leq \phi(y) \leq \hat{C}_2$ for all
			$y \in [\eta, \theta]$.
		\end{enumerate}
		On the other hand, under the conditions of Lemma
		\ref{lem:doeblinmitosis} we know again that $\phi(x)$ is bounded
		above and below by positive constants in each interval of the form
		$(0, \theta]$.
		
		Therefore we obtain for $x_0 \in [\eta,\theta]$:
		\[ \mathcal{F}_t \delta_{x_0}(x)
		= \frac{\phi(x)}{\phi(x_0)} \mathcal{S}_t \delta_{x_0} (x)
		\geq
		\frac{\hat{C}_1(\eta, \theta)}{\hat{C}_2(\eta, \theta)}
		C(\eta, \theta, t) := \breve{C}(\eta, \theta, t),
		\]
		allowing $\eta=0$ if $\int_0^1 1/g < +\infty$.
	\end{proof}

	\begin{proof}[Proof of Theorem \ref{thm:main}]
		As remarked above, the semigroup $(\mathcal{F}_t)_{t \geq 0}$
		satisfies the Lyapunov condition in Hypothesis~\ref{hyp:Lyapunov} in
		all cases, for $t \geq 1$, with a weight $V$ and constants $\gamma$,
		$K$ which are independent of $t$. In order to satisfy
		Hypothesis~\ref{hyp:localDoeblin} it is enough then to find any time
		$t \geq 1$ for which we have a uniform lower bound whenever the initial
		condition is a delta function supported on a region of the form
		\begin{equation*}
			\mathcal{C} := \left \{x > 0 \mid V(x) \leq R \right \}
		\end{equation*}
		for some $R > 2K / (1- \gamma)$. Lemma \ref{lem:doeblingeneral}
		gives this in all cases. Notice that in the cases in which the lower
		bound is only available for $x_0 \in [\eta, \theta]$ with $\eta >
		0$, the function $V$ we give in Section \ref{sec:lyapunov} is
		unbounded at $x=0$, and thus the region $\mathcal{C}$ is contained
		in an interval of that form.  
	\end{proof}
	
	\paragraph{\bf Explicit calculations for the self-similar fragmentation case.}
	We recall that the so-called self-similar fragmentation equation corresponds to a linear growth rate \mbox{$g(x)=x$}, a monomial total fragmentation rate $B(x)=x^b$, $b>0$, and a self-similar kernel (here we take the homogeneous self-similar kernel $p(z)\equiv2$).
	In that case, all the constants appearing in Harris's theorem can be quantified.
	This is due to the explicit expression $\phi(x)=x$ of the dual eigenfunction when $g(x)=x$.
	For the computations we choose for instance the parameters $k=0$ and $K=2$, which correspond to the Lyapunov function $V(x)=(x^k+x^K)/\phi(x)=1/x+x$.
	We start with Hypothesis~\ref{hyp:Lyapunov}.
	Using that $B(x)=x^b$ we can make the proof of Lemma~\ref{lyapunovBlin} more quantitative.
	Indeed the function $\Phi$ defined in~\eqref{Phi_lin} reads in the present case
	\[\Phi(x)=-\frac13x^{b+1}+\frac32x+x^{b-1}-\frac12x^{-1}.\]
	Treating separately the cases $x\leq1$, $x\geq1$, and $b\leq2$, $b\geq2$, we can check that
	\[\Phi(x)\leq-\frac13x^{b+1}+\frac32x+x^{b-1}\leq5\Big(\frac{15}{2}\Big)^{\frac1b+\frac b2}\]
	for all $x>0$.
	So Hypothesis~\ref{hyp:Lyapunov} is verified for any $t_0>0$ with the constants
	\[\gamma=e^{-\frac{t_0}{2}}\qquad\text{and}\qquad K=10\Big(\frac{15}{2}\Big)^{\frac1b+\frac b2}.\]
	We now turn to Hypothesis~\ref{hyp:localDoeblin}.
	We choose
	\[R=\frac{4K}{1-\gamma}\]
	and we notice that since $V(x)=1/x+x$
	\[\mathcal C=\left \{ x:V(x)\leq R \right \}\subset[1/R,R].\]
	For $\phi(x)=x$ and $p(z)\equiv2$, Equation~\eqref{eq:gfconserv} reads
	\[\frac{\p }{\p t}f(t,x) + \frac{\p }{\p x}\left( xf(t,x) \right) + B(x) f(t,x) = 2 \int_{x}^{+ \infty}  B(y)f(t,y)\frac{x}{y} \d y\]
	and we can prove directly on this equation, proceeding similarly as in Lemma~\ref{lem:doeblinuniformfrag}, that for any $t_0>0$ and all $x_0\in[1/R,R]$
	\[\mathcal F_{t_0}\delta_{x_0}\geq \alpha\nu\]
	with
	\[\nu(\d y)=\frac{2e^{-2t_0}}{R}\1_{[0,Re^{t_0}]}(y)y\d y\qquad\text{and}\qquad \alpha=R^{b+3}t_0\exp\Big(-2R^\gamma \frac{e^{b t_0}}{b}\Big).\]
	We are now in position to apply Harris's theorem.
	Choosing in Theorem~\ref{thm:Harris}
	\[\alpha_0=\frac\alpha2\qquad\text{and}\qquad\gamma_0=\gamma+\frac{2K}R\]
	we obtain
	\[\bar\alpha=\max\left\{1-\frac\alpha2,\frac{1-\gamma+\frac{1+\gamma}{2}\alpha}{1-\gamma+\alpha}\right\}.\]
	Choosing $t_0=2\log2$ we get
	\[\gamma=\frac12,\quad R=80\Big(\frac{15}{2}\Big)^{\frac1b+\frac b2},\quad\alpha=2\log2 R^{b+3}e^{-2(4R)^b/b}\]
	and
	\[\bar\alpha=\max\left\{1-\frac\alpha2,1-\frac{\alpha}{2(1+2\alpha)}\right\}=1-\frac{\alpha}{2(1+2\alpha)}.\]
	This proves that we can choose $\rho$ as in~\eqref{rho_num}.
	
	\section*{Acknowledgements}
	
	JAC and HY were supported by the project MTM2017-85067-P, funded by
	the Spanish government and the European Regional Development Fund. HY
	was also supported by the Basque Government through the BERC 2018-2021
	program, by the Spanish Ministry of Economy and Competitiveness
	MINECO: BCAM Severo Ochoa excellence accreditation SEV-2017-0718, by the ``la Caixa'' Foundation and by the European Research Council (ERC)
	under the European Union’s Horizon 2020 research and innovation
	programme (grant agreement No 639638). JAC and HY gratefully
	acknowledge the support of the Hausdorff Research Institute for
	Mathematics (Bonn), through the Junior Trimester Program on Kinetic
	Theory.
	PG was supported by the ANR project NOLO, ANR-20-CE40-0015, funded by the French Ministry of Research.
	
	\
	
	\bibliography{bibliography}
	
\end{document}